\documentclass[11pt]{article}

\usepackage{tikz}
\usepackage{answers}
\usepackage{xcolor}
\usepackage{setspace}
\usepackage{graphicx}
\usepackage{multicol}
\usepackage{mathrsfs}
\usepackage[margin=1in]{geometry} 
\usepackage{amsmath,amsthm,amssymb}
\usepackage{mathtools}
\usepackage[utf8]{inputenc}
\usepackage[english]{babel}
\usepackage{bbm}
\usepackage{breqn}
\usepackage{subfiles}
\usepackage[english]{babel}
\usepackage[nottoc]{tocbibind}
\usepackage{bm}
\usepackage{hyperref}
\usepackage{algorithm}
\usepackage{subcaption}
\usepackage{paralist}
\usepackage{thm-restate}

\newcommand{\citep}{\cite}
\newtheorem{theorem}{Theorem}[section]

\newtheorem{lemma}[theorem]{Lemma}

\newtheorem{assumption}[theorem]{Assumption}

\DeclarePairedDelimiter\floor{\lfloor}{\rfloor}

\renewcommand{\S}{\mathbb S}

\newcommand{\R}{\mathbb{R}}

\DeclareMathOperator*{\E}{\mathbb{E}}

\newcommand{\Var}{\text{Var}}
\newcommand{\sign}{\text{sign}}

\newcommand{\eps}{\varepsilon}
\newcommand{\1}{\mathbbm{1}}

\newcommand{\inner}[1]{\langle#1\rangle}

\newcommand{\wh}{\widehat}

\DeclarePairedDelimiterX{\abs}[1]{|}{|}{#1}
\DeclarePairedDelimiterX{\norm}[1]{\lVert}{\rVert}{#1}

\DeclarePairedDelimiterX{\kl}[2]{D_{KL}(}{)}{%
  #1\;\delimsize\|\;#2%
}

\hypersetup{
    colorlinks=true,
    linkcolor=blue,
    filecolor=magenta,      
    urlcolor=cyan,
}

\title{Beyond Catoni: Sharper Rates for Heavy-Tailed and Robust Mean Estimation}
\author{
    Shivam Gupta\\
    UT Austin\\
    \texttt{shivamgupta@utexas.edu}\\
    \and
    Samuel B. Hopkins\\
    MIT\\
    \texttt{samhop@mit.edu}\\
    \and Eric Price\\
    UT Austin\\
    \texttt{ecprice@cs.utexas.edu}
    }
\begin{document}
\maketitle
\thispagestyle{empty}

\begin{abstract}%
        We study the fundamental problem of estimating the mean of a $d$-dimensional distribution with covariance $\Sigma \preccurlyeq \sigma^2 I_d$ given $n$ samples. When $d = 1$, \cite{catoni} showed an estimator with error $(1+o(1)) \cdot \sigma \sqrt{\frac{2 \log \frac{1}{\delta}}{n}}$, with probability $1 - \delta$, matching the Gaussian error rate. For $d>1$, a natural estimator outputs the center of the minimum enclosing ball of one-dimensional confidence intervals to achieve a $1-\delta$ confidence radius of $\sqrt{\frac{2 d}{d+1}} \cdot \sigma \left(\sqrt{\frac{d}{n}} + \sqrt{\frac{2 \log \frac{1}{\delta}}{n}}\right)$, incurring a $\sqrt{\frac{2d}{d+1}}$-factor loss over the Gaussian rate. When the $\sqrt{\frac{d}{n}}$ term dominates by a $\sqrt{\log \frac{1}{\delta}}$ factor, \cite{lee2022optimal-highdim} showed an improved estimator matching the Gaussian rate. This raises a natural question: Is the $\sqrt{\frac{2 d}{d+1}}$ loss \emph{necessary} when the $\sqrt{\frac{2 \log \frac{1}{\delta}}{n}}$ term dominates?

    We show that the answer is \emph{no} -- we construct an estimator that improves over the above naive estimator by a constant factor. We also consider robust estimation, where an adversary is allowed to corrupt an $\eps$-fraction of samples arbitrarily: in this case, we show that the above strategy of combining one-dimensional estimates and incurring the $\sqrt{\frac{2d}{d+1}}$-factor \emph{is} optimal in the infinite-sample limit.
\end{abstract}

\newpage

\thispagestyle{empty}

\tableofcontents

\newpage

\setcounter{page}{1}

\section{Introduction}
Mean estimation is perhaps the simplest statistical estimation
problem: given samples $x_1,\ldots,x_n \sim D$ for some $d$-dimensional probability
distribution $D$, estimate the mean $\mu$ of $D$.  If $x$ is
Gaussian with covariance $\Sigma \preceq \sigma^2 I_d$, then the empirical mean
is the optimal estimator.  It satisfies
\begin{align}\label{eq:gaussianerror}
  \norm{\wh{\mu} - \mu} \leq \sigma \left(\sqrt{\frac{d}{n}} + \sqrt{\frac{2 \log\frac{1}{\delta}}{n}}\right)
\end{align}
with probability $1-\delta$.  Even if $x$ is not Gaussian, for any
fixed ($D, d, \delta$), as $n \to \infty$ the central limit theorem
shows that the empirical mean achieves the Gaussian
rate~\eqref{eq:gaussianerror}.
But when the distribution, dimension,
or failure probability can vary with $n$, more sophisticated
estimators are needed to get good rates.  If the distribution has
outliers---large, rare events---the empirical mean can perform
very badly.

In \emph{one} dimension, the Median-of-Means estimate \citep{nemirovskij1983problem,jerrum1986random,alon1996space} is the classic way to
get subgaussian rates with minimal assumptions on the distribution. For
\emph{any} 1-dimensional distribution $D$ of variance
$\sigma^2$, the median (over $\Theta(\log \frac{1}{\delta})$ batches)
of means (of $\Theta\left(\frac{n}{\log \frac{1}{\delta}}\right)$ samples per batch) satisfies
\[
  \abs{\wh{\mu} - \mu} \leq O\left( \sigma \cdot \sqrt{\frac{\log\frac{1}{\delta}}{n}} \right)
\]
with $1-\delta$ probability, i.e., it achieves the Gaussian
rate~\eqref{eq:gaussianerror} up to constant factors.  But such constants are important in statistical estimation: statistics
texts, for example \citep{maindonald_braun_2010,wasserman2004all,wackerly2014mathematical,casella2021statistical}, discuss asymptotic relative efficiency of the mean over the median (and asymptotic optimality of maximum-likelihood estimators in general) as an important consideration in choosing an estimator---in this case, the asymptotic efficiency of the mean results in a $\sqrt{\frac{2}{\pi}}$ factor smaller error bound in
the Gaussian case, leading to $\approx 36\%$ lower sample complexity.  As a
result, many practitioners use the mean, and then are vulnerable to
outliers.  It is therefore important to have estimators that are as
efficient as possible, while still working without strong assumptions
on the data distribution.

To address this, \cite{catoni} developed
a 1-dimensional mean estimator that is tight up to $1 + o(1)$ factors:
for $n \gg \log \frac{1}{\delta}$, it gives error
\[
  \abs{\wh{\mu} - \mu} \leq (1+ o(1))\cdot \sigma \sqrt{\frac{2\log\frac{1}{\delta}}{n}},
\]
matching the Gaussian rate~\eqref{eq:gaussianerror}.  Catoni's
estimator requires knowledge of the variance $\sigma^2$; this
requirement was removed by \cite{lee2022optimal}, at a cost of a
larger $o(1)$ term.
Even the Median-of-Means-style $O(\sigma \cdot \sqrt{\log \frac 1 \delta / {n}}))$ guarantee is information-theoretically impossible if $n \ll \log \frac 1 \delta$ \citep{devroye_subg_multiple_delta}.
It is open whether the Catoni-style $(1+o(1))$ guarantee can be achieved when $n = \Theta(\log \frac 1 \delta)$.
We henceforth assume $n \gg \log \frac 1 \delta$.


\paragraph{High-dimensional mean estimation.}  In
dimension $d > 1$, naively applying a 1-dimensional estimator to the
coordinates independently leads to the suboptimal rate
$O\left(\sigma \cdot \sqrt{\frac{d \log \frac{d}{\delta}}{n}}\right)$.  Over the past
few years, a number of works in statistics and theoretical computer
science have developed better estimators~\citep{lugosi2019sub,hopkins_subgaussian,pmlr-v99-cherapanamjeri19b}, matching the
Gaussian rate~\eqref{eq:gaussianerror} up to constant factors.  But as
with $d=1$, we can ask: what constant factors are achievable, and in particular, can the Gaussian rate~\eqref{eq:gaussianerror} be matched up to $(1+o(1))$?

There are two terms in~\eqref{eq:gaussianerror}, and so we will refer to two different constants: the optimal constant $c_{d}$ on
$\sqrt{\frac{d}{n}}$ whenever $d \gg \log \frac{1}{\delta}$ and the
first term dominates, and the optimal constant $c_{\delta}$ on
$\sqrt{\frac{2 \log \frac{1}{\delta}}{n}}$ when
$\log \frac{1}{\delta} \gg d$ and the second term dominates. There is
also a third regime---when $d \eqsim \log \frac{1}{\delta}$---but this
regime is quite complicated to analyze.  Even in the Gaussian case,
the error bound~\eqref{eq:gaussianerror} does not give the tight
constant in this regime.  For this paper we ignore the intermediate
regime.

One can get a natural upper bound on these constants by lifting $1$-dimensional estimators to
$d$ dimensions. \cite{catoni_giulini} used a ``PAC-Bayes'' argument to show that if the Catoni estimator is applied
to every direction $u$, then every estimate $\wh{\mu}_u$ of
$\inner{\mu, u}$ has error bounded by the Gaussian
rate~\eqref{eq:gaussianerror}.  The set of possible $d$-dimensional
means $\mu$ that satisfy all these $1$d constraints has diameter twice
this error rate.  One can then output the center $\wh{\mu}$ of the
minimum enclosing ball of this set. Jung's theorem~\citep{Jung1901} states
that this loses just a constant factor: any set of diameter $2$ is
enclosed in a ball of radius
$JUNG_d := \sqrt{\frac{2d}{d+1}} \leq \sqrt{2}$. Therefore
\begin{align}
    \label{eq:jung_d_error}
  \norm{\hat{\mu} - \mu} \leq JUNG_d \cdot (1+o(1)) \sigma \left( \sqrt{\frac{d}{n}} + \sqrt{\frac{2 \log\frac{1}{\delta}}{n}} \right )
\end{align}
and so both $c_{d}$ and $c_{\delta}$ are at most
$JUNG_d \leq \sqrt{2}$.  For very large dimension one can do better:
\cite{lee2022optimal-highdim} showed for $d \gg \log^2 \frac{1}{\delta}$
that the Gaussian rate~\eqref{eq:gaussianerror} can be matched
precisely, so $c_d = 1$ for such large $d$.

\paragraph{Our contributions: heavy-tailed estimation.}
%
%


Our main result gives an algorithm with a strictly better constant factor than in \eqref{eq:jung_d_error} when $\log \frac{1}{\delta} \gg d$ and $d \ge 2$---that is, we show that $c_\delta$ is strictly smaller than $JUNG_d$ for all $d \geq 2$.

\begin{restatable}{theorem}{heavyupperthm}\label{thm:heavy_upper}
  There exists constants $\tau, C > 0$ such that the following holds.  Let $d \geq 2$, and suppose $n \geq C \log \frac{1}{\delta} \geq C^2 d$.
  There is an algorithm that takes $n$ samples from a distribution over $\R^d$ with covariance $\Sigma \preceq \sigma^2 I$, as well as $\sigma^2$ and $\delta$, and outputs an estimate $\wh \mu$ of the mean $\mu$ that achieves
  \[
    \norm{\wh{\mu} - \mu} \leq (1-\tau) \cdot JUNG_d \cdot \sigma \sqrt{\frac{2 \log \frac{1}{\delta}}{n}}
  \]
  with $1-\delta$ probability.
\end{restatable}
In particular, the limiting constant as $d \to \infty$ is
$\sqrt{2} - \tau$ for some $\tau > 0$. 

\paragraph{Our contributions: robust estimation.}  A related problem,
also extensively studied in theoretical computer science over the past
decade, is \emph{robust} mean estimation~\citep{diakonikolas2023algorithmic}.  In robust mean estimation, the data is
initially drawn from a covariance $\Sigma \preceq I$ distribution, but
an adversary can corrupt an arbitrary $\eps$ fraction of the data
points.  In this model, estimation error remains even in the
population limit as $n \to \infty$.  In one dimension, the optimal
error bound is
\[
  (1 + O(\eps))\sqrt{2 \eps}.
\]

As with heavy-tailed estimation, one can lift the 1d estimator to
higher dimensions: apply the one-dimensional estimator in every
direction, take the intersection of their confidence intervals to get
a set of candidate means, and output the center of the minimum
enclosing ball.  And as with heavy-tailed estimation, this loses a
factor $JUNG_d = \sqrt{\frac{2d}{d+1}}$.  But,
unlike with heavy-tailed estimation, this is tight:

\begin{restatable}{theorem}{robustlowerthm}\label{thm:robustlower}
  For every $d \geq 1$ and $\eps \leq \frac{1}{2}$, every algorithm
  for robust estimation of $d$-dimensional distributions with
  covariance $\Sigma \preceq \sigma^2 I$ has error rate
  \[
    \E[\norm{\wh{\mu} - \mu}] \geq JUNG_d \cdot (1 + O(\eps)) \cdot \sqrt{2 \sigma^2 \eps}
  \]
  on some input distribution, in the population limit.
\end{restatable}

As discussed above, this is matched by the (somewhat folklore)
algorithm of estimating all $1$d projections and taking the center of the
minimum enclosing ball of feasible means:

\begin{restatable}[Folklore + Jung's theorem]{theorem}{robustupperthm}\label{thm:robustupper}
  For every $d \geq 1$ and $\eps \leq \frac{1}{3}$, there is an
  algorithm for robust estimation of $d$-dimensional distributions
  of covariance $\Sigma \preceq \sigma^2 I$ with error rate
  \[
    \E[\norm{\wh{\mu} - \mu}] \leq JUNG_d \cdot (1 + O(\eps)) \cdot \sqrt{2 \sigma^2 \eps}
  \]
  in the population limit.
\end{restatable}

We provide the full proof of Theorem~\ref{thm:robustlower} in Appendix~\ref{sec:robust-lower} and
Theorem~\ref{thm:robustupper} in Appendix~\ref{sec:robust-upper}.

\paragraph{Summary.} The mean estimation error bound has three terms,
corresponding to the dependence on dimension $d$, on failure
probability $\delta$, and on robustness $\eps$.  Lee and
Valiant~\cite{lee2022optimal-highdim} showed that the $d$-dependent term does not lose a
constant factor relative to the Gaussian rate, for sufficiently large
$d$.  We show that the $\eps$-dependent term loses exactly the
constant $JUNG_d$ that arises when lifting $1$-dimensional estimates
to $d$-dimensional estimates, while the $\delta$-dependent term is better than
$JUNG_d$ times the Gaussian rate, for all $d \neq 1$.
For the latter result, we construct a novel high-dimensional mean estimator which goes beyond lifting a one-dimensional estimator.

\subsection{Related Work}
\paragraph{Heavy-tailed and Robust Estimation.} Both settings been extensively studied by the statistics and theoretical computer science communities; see for example, a recent survey \citep{lugosi2019mean} and book \citep{diakonikolas2023algorithmic}. For heavy-tailed estimation, several works have established asymptotic bounds matching the Gaussian rate for a variety of estimation tasks, including mean estimation~\citep{lugosi2019sub,catoni_giulini}, covariance estimation~\citep{abdalla2023covariance,Mendelson2018RobustCE}, and regression~\citep{heavy_tail_regression}. Similarly, robust estimation has been studied in a variety of settings, including mean estimation~\citep{robust_first,diakonikolas2017robustly}, covariance estimation~\citep{robust_covariance_cheng}, list-decodable estimation~\citep{list_decodable_robust,list_decodable_robust_fast}, and regression~\citep{robust_regression}.
\citep{diakonikolas2020outlier,hopkins2020robust} study rigorous connections between robust and heavy-tailed estimation.

Despite the large body of work on both these models, the algorithms proposed have so far seen limited adoption in practice. One reason for this is suboptimal constants. Samples can be precious, and statistics texts often report ``asymptotic relative efficiency'' of various estimators (similar in spirit to the constant factors we study here).
Since the empirical mean has optimal asymptotic efficiency,
in some texts practitioners are taught to use the mean over the median (despite the robustness the median provides) if the data ``looks'' Gaussian via eyeballing \citep{maindonald_braun_2010}, since using the median would require collecting $\approx 50\%$ more samples.
In one dimension, this is unprincipled and error prone; in high dimensions, it is not even a viable strategy.

\paragraph{Towards optimal constants.} To overcome the above issues and promote adoption, there has been a flurry of recent work attempting to achieve sharp rates (including constants) for a variety of statistical estimation \citep{lee2022optimal, lee2022optimal-highdim, pmlr-v195-minsker23a, Minsker2022UstatisticsOG, catoni,catoni_giulini,devroye_subg_multiple_delta,GLP_high_dim,GLP_symmetric,gupta2023minimaxoptimal} and testing \citep{GP_huber,dang2023improving, kipnis2023minimax} tasks. Of these, for heavy-tailed estimation, Catoni~\citep{catoni} showed an estimator matching the Gaussian rate in dimension $d=1$ when the variance $\sigma^2$ is known. This was followed by work that achieved the same rate even when $\sigma^2$ is unknown~\citep{lee2022optimal}. 

For $d>1$, a natural estimator outputs the center of the minimum enclosing ball of the intersection of one-dimensional confidence intervals. For covariance $\Sigma \preccurlyeq \sigma^2 I_d$, \cite{catoni_giulini} showed a PAC-Bayes argument that implies a $\sqrt{\frac{2d}{d+1}} \cdot \sigma \left(\sqrt{\frac{d}{n}} + \sqrt{\frac{2 \log \frac{1}{\delta}}{n}}\right)$ rate for this estimator, incurring a $\sqrt{\frac{2 d}{d+1}}$ factor over the Gaussian rate. When the $\sqrt{\frac{d}{n}}$ term dominates by a $\sqrt{\log \frac{1}{\delta}}$ factor, \cite{lee2022optimal-highdim} showed an estimator with an improved rate of $\sigma \sqrt{\frac{d}{n}}$, matching the Gaussian rate in this regime. This work shows that the $\sqrt{\frac{2d}{d+1}}$ factor can be improved upon even when the $\sqrt{\frac{2 \log \frac{1}{\delta}}{n}}$ term dominates.




\section{Proof Overview}


\subsection{Heavy-Tailed Estimator}

\paragraph{High-level goal.} In one dimension, the optimal error rate for $(1-\delta)$-probability mean estimation is
$\sigma \sqrt{\frac{2\log \frac{1}{\delta}}{n}}$, which we will call
$OPT_1$.  In $d$ dimensions, one can apply the one-dimensional bound
in every direction (with either a union bound, or more efficiently
with PAC-Bayes~\citep{catoni_giulini}) to identify a set of candidate means of
diameter $2 OPT_1 + O(\sigma \sqrt{d /n})$; suppose
$\log \frac{1}{\delta} \gg d$, so the high-probability term $2 OPT_1$
dominates.  Then, Jung's theorem states that the minimum enclosing ball of this
set has radius at most
$\sqrt{\frac{2d}{d+1}} \cdot OPT_1  = JUNG_d \cdot OPT_1$.
In Theorem~\ref{thm:heavy_upper} we show that a better constant factor
is possible.

Our key technical result is a mean estimation algorithm for \emph{two}
dimensions, with error
$(1 - \tau)\cdot JUNG_2 \cdot OPT_1 = (1-\tau)\frac{2}{\sqrt{3}} \cdot
OPT_1$ for a constant $\tau > 0$.  Given this result, we can lift it
to higher dimensions using a generalization of Jung's
theorem~\citep{generalized_Jung}: for a dimension-$d$ set $S$, if
every dimension-$k$ projection has length $2r_k$, then $S$ is enclosed
in a ball of radius $r_k \cdot \frac{JUNG_d}{JUNG_k}$.  So our
$(1-\tau)$ improvement for $d=2$ yields a $(1-\tau)$ improvement for
all $d$, and in particular asymptotic error $(1-\tau)\sqrt{2} \cdot OPT_1$
rather than $\sqrt{2} \cdot OPT_1$ for $d \to \infty$.

\paragraph{Variant of Catoni's estimator for $d=1$.}
To understand our $d=2$ estimator, it's helpful to understand how to
get the optimal constant for $d=1$.  In Appendix~\ref{sec:catoni} we
give a simple, 2-page self-contained analysis of a variant of Catoni's
estimator~\citep{catoni}.

Define $T = \sigma \sqrt{\frac{n}{2 \log \frac{1}{\delta}}}$, and
consider a $\psi$ function satisfying
\begin{align}
  \label{eq:catonirequirement}
-\log\left(1 - x + \frac{x^2}{2}\right)\le \psi(x) \le \log\left(1 + x + \frac{x^2}{2} \right)  
\end{align}
such as $\psi(x) = x - x^3/6$ for $\abs{x} \leq \sqrt{2}$, and
$\psi(x) = \frac{2 \sqrt{2}}{3} \cdot \sign(x)$ otherwise. We plot this function below, along with two other functions from~\citep{catoni} satisfying the above bound.

\begin{figure}[H]
    \centering
    \hspace{-0.8cm}
    \includegraphics[width=0.5\textwidth]{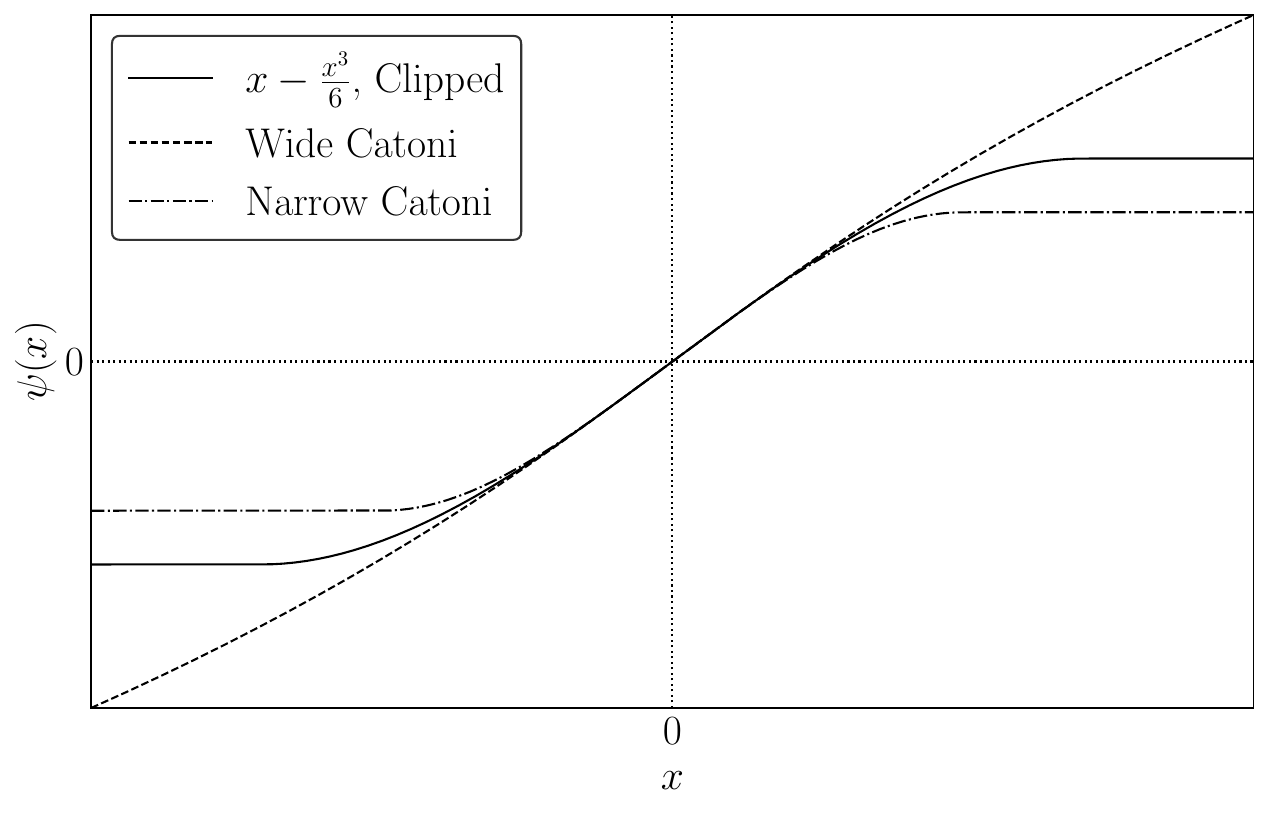}
    \caption{Some $\psi$ functions satisfying Catoni's constraints~\eqref{eq:catonirequirement}} 
    \label{fig:catoni_psi}
\end{figure}

Suppose we have an initial estimate $\mu_0$ that has small big-O 
error, but with a large constant factor---say, the median-of-means
estimate on an initial sample of $\xi n$ points for a small constant
$\xi$.  This will have error
$O_\xi\left(\sigma \sqrt{\frac{\log \frac{1}{\delta}}{n}}\right)$, which we
would like to drive down to
$OPT_1 = \sigma \sqrt{\frac{2\log \frac{1}{\delta}}{n}}$.  The final
estimate is
\begin{align}\label{eq:catoniest}
  \wh{\mu} = \mu_0 + \frac{1}{n} \sum_{i=1}^n T \psi\left(\frac{x_i - \mu_0}{T}\right)
\end{align}
Intuitively, $T$ is the threshold for being an outlier: if
$\abs{x} \ll T$ always, then Bernstein's inequality will give that the
empirical mean achieves $(1 + o(1)) OPT_1$.  And indeed,
$T \psi(x / T) \approx x$ for $\abs{x} \ll T$, so the
estimate~\eqref{eq:catoniest} is close to the empirical mean in this
case.  On the other hand, elements $\abs{x} \gg T$ will only be
sampled $o(\log \frac{1}{\delta})$ times by Chebyshev's inequality, so
the sample of such events is completely unreliable for $1-\delta$
failure probability; the influence of such elements on the
estimator~\eqref{eq:catoniest} is negligible.  The challenge is to
handle the cases of $\abs{x} = \Theta(T)$.

The natural approach to show that $\wh{\mu}$ concentrates about $\mu$
is to bound its moment generating function (MGF).  The
conditions~\eqref{eq:catonirequirement} are precisely what are needed: $\E[\exp(\frac{n}{T} \wh{\mu})]$ depends on $\E[\exp(\psi( (x-\mu)/T))]$, which is controlled by just the mean and variance of $x$ through~\eqref{eq:catonirequirement}.  As we show in
Lemma~\ref{lem:vanilla_catoni_local}, this leads to the concentration bound
\[
  \abs{\wh{\mu} - \mu} \leq \left(1 + O\left(\frac{\log \frac{1}{\delta}}{n}\right)\right) OPT_1
\]
with probability $1-\delta$.

The estimator~\eqref{eq:catoniest} we analyze in Appendix~\ref{sec:catoni} is different from
the original Catoni estimator in that Catoni finds a root of
$\psi(\frac{x - \mu}{T})$, while our variant approximates this root
with essentially one step of Newton's method.  Our analysis does not
handle reuse of samples, so it requires the initial estimate $\mu_0$
to use a small initial sample.  This makes our analysis 
simpler than~\citep{catoni}, which is helpful for the extension we need to get the better
constant for $d = 2$.

\paragraph{A better constant for ``inlier-light'' distributions.}
The error of the estimate $\wh{\mu}$ is bounded by the
constraints~\eqref{eq:catonirequirement}.  So with a \emph{better}
bound, the estimate would sharpen by a constant factor.
In particular, if
we could find a $\psi$ with
\begin{align}
  \label{eq:catonirequirement2}
-\log\left(1 - x + (1-\eta)\frac{x^2}{2}\right)\le \psi(x) \le \log\left(1 + x + (1-\eta)\frac{x^2}{2} \right)  
\end{align}
then the variance term which appears in the MGF argument above would have a leading $(1 - \eta)$
factor, giving a better constant.  Unfortunately, \eqref{eq:catonirequirement} is not achieved by any function $\psi$ for all $x$ simultaneously: both the upper and lower constraints~\eqref{eq:catonirequirement} were
$x - x^3/6 \pm \Theta(x^4)$, so for any $\eta > 0$ if $x$ is small enough, shifting the constraints closer by
$\Theta(x^2)$ is impossible.  

But, for any $\beta > 0$, if we restrict attention to $x$ such that $\abs{x} > \beta$, the constraints of~\eqref{eq:catonirequirement} do not exactly
match, so there exists an $\eta$ for
which~\eqref{eq:catonirequirement2} is possible for all $|x| > \beta$.
We have already discussed one function satisfying tightened constraint: $\psi(x) = x -x^3/6$ for $|x| \le \sqrt{2}$ and $\psi(x) = \left(\sqrt{2} - \frac{2 \sqrt{2}}{6} \right) \sign(x)$ otherwise. This is plotted in Figure~\ref{fig:catoni_psi}.

As a result of this improved analysis, the Catoni estimate~\eqref{eq:catoniest} is a constant factor
better at handling the variance caused by $x$ whenever
$\abs{x} \gtrsim T$.

To formalize this idea, for any constants $\beta, L > 0$, we say a distribution is
``$(\beta, L)$-inlier-light'' if it has at most $(1-L) \sigma^2$
variance from elements smaller than $\beta T$. Catoni gets the tight
constant on the $(1-L) \sigma^2$ variance from inliers, and a
\emph{better} constant on the remaining at-most-$L \sigma^2$ variance.
Thus it gets error $(1 - \tau) OPT_1$ error on inlier-light
distributions, for some constant $\tau$ depending on $\beta$ and $L$.

\paragraph{An alternative to Catoni for outlier-light distributions.}
On the other hand, if a distribution is \emph{not} inlier-light, it
can have very few outliers: there's at most $L \sigma^2$ variance
remaining to come from outliers.  If we trim at a threshold $\alpha T$ for $\alpha > \beta$, then the contribution to the mean from the trimmed outliers is small:
the worst-case is when they are all at the threshold $\alpha T$, in which case the contribution is
$\frac{L \sigma^2}{\alpha^2 T^2} \cdot \alpha T = \frac{L}{\alpha}
OPT_1$.  And for small $\alpha$, Bernstein's
inequality says that the empirical mean of the untrimmed inliers will have
accuracy $(1 + O(\alpha))OPT_1$.  As a result, the
\emph{trimmed} mean, trimmed to $\alpha T$, achieves
$(1 + O(\alpha + L/\alpha)) OPT_1$ on distributions that are \emph{not} $(\beta, L)$
inlier-light, for any $\alpha > \beta$.

Note also that the property of being inlier-light can be tested with $1-\delta$ accuracy, since
it involves measuring the variance from bounded entries, as long as $L \gtrsim \beta$.  So for $L = \Theta(\beta)$, we can (1) test for inlier-lightness, and on non-inlier-light distributions (2) trim at $\alpha = \sqrt{\beta}$ to get $(1 + O(\sqrt{\beta})) OPT_1$ error.

\paragraph{Handling $d=2$.}
Per the above, in one dimension \emph{either} the Catoni estimate
achieves a constant better than 1, \emph{or} the trimmed mean achieves
constant close to 1.  The latter is promising because the empirical mean, in the
subgaussian case where it works, gets error $OPT_1$ \emph{independent
  of the dimension}.

Our $d=2$ algorithm is as follows.  We test whether the distribution
is inlier-light in either direction $e_1$ or $e_2$; if it is, we run
Catoni on every $1$d projection in a fine net around the circle, and
take the center of the minimum enclosing ball of the possible means.
In general, this gets at most $JUNG_2 \cdot OPT_1$ error; but the
tight instance for Jung is an equilateral triangle, and this error
only happens if Catoni gets error bound $OPT_1$ in three directions
approximately $120^\circ$ apart.  If the distribution is inlier-light
in some direction $e_i$, then it is also inlier-light (with slight
loss in parameters) in at least one of the triangle directions, so
Catoni gets a better error in that direction and a more accurate
estimate overall.

On the other hand, if the distribution is not inlier-light in either
the $e_1$ or $e_2$ direction, we remove any element larger than
$\sqrt{\beta} T$ in either direction and take the empirical mean of all other
samples.  This gets error $(1 + O(\sqrt{\beta})) OPT_1$, without
any dependence on $JUNG_2$.

For small enough constant $\beta$ and $L = \Theta(\beta)$, either situation will give a
constant better than $JUNG_2$. Finally, as stated before, we can lift our two-dimensional estimate to higher dimensions using a generalization of Jung's theorem (Theorem~\ref{thm:generalized_jung}, \cite{generalized_Jung}) to obtain a constant better than $JUNG_d$ in $d$-dimensions.

\subsection{Robust Estimation, Lower Bound}

Now, we discuss the ideas behind Theorem~\ref{thm:robustlower}, showing that the naive strategy of combining one-dimensional estimates is optimal for the robust estimation setting.


\begin{figure}[H]
    \vspace{-0.2cm}
    \centering
    \begin{minipage}{.23\textwidth}
    \includegraphics[width=\textwidth]{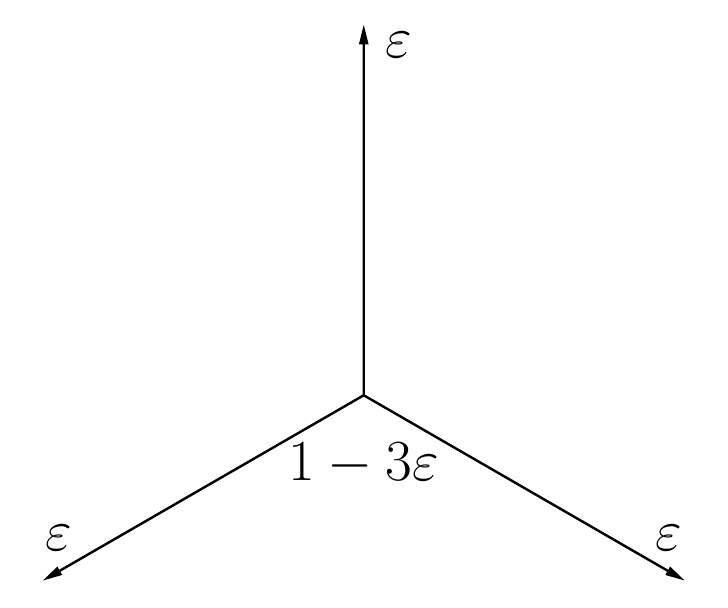}
    \end{minipage}
    \begin{minipage}{.23\textwidth}
       \includegraphics[width=\textwidth]{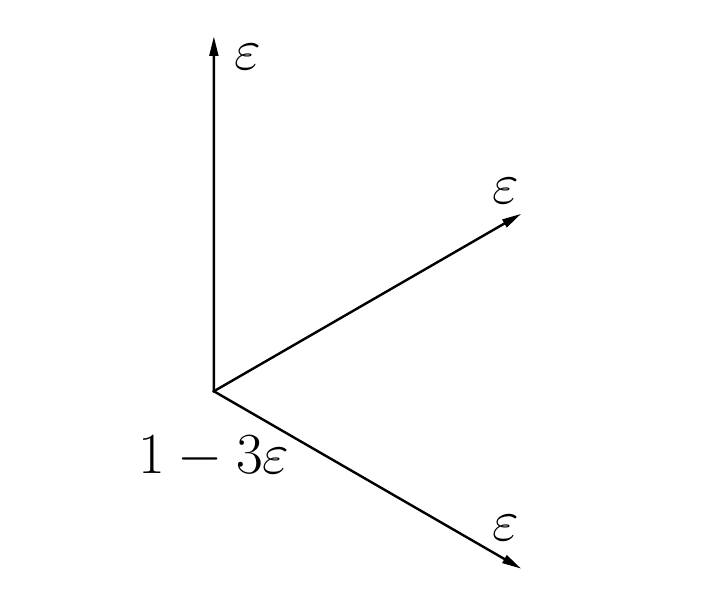} 
    \end{minipage}
    \begin{minipage}{.23\textwidth}
       \includegraphics[width=\textwidth]{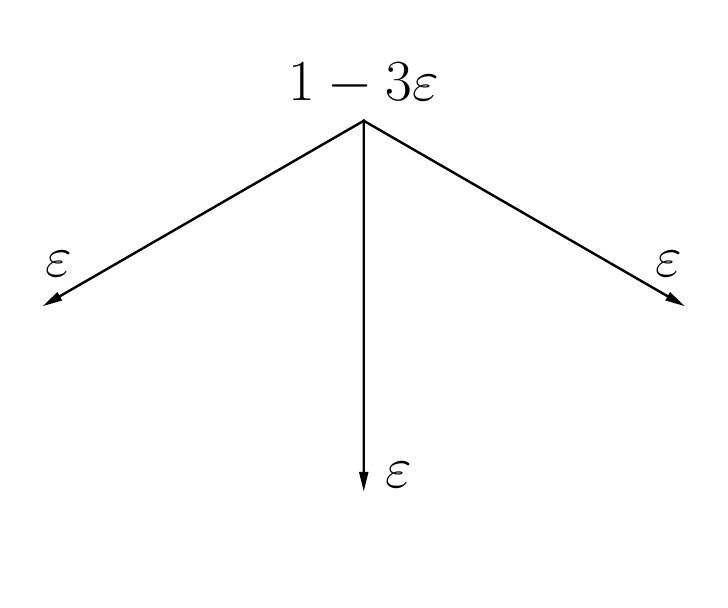} 
    \end{minipage}
    \begin{minipage}{.23\textwidth}
       \includegraphics[width=\textwidth]{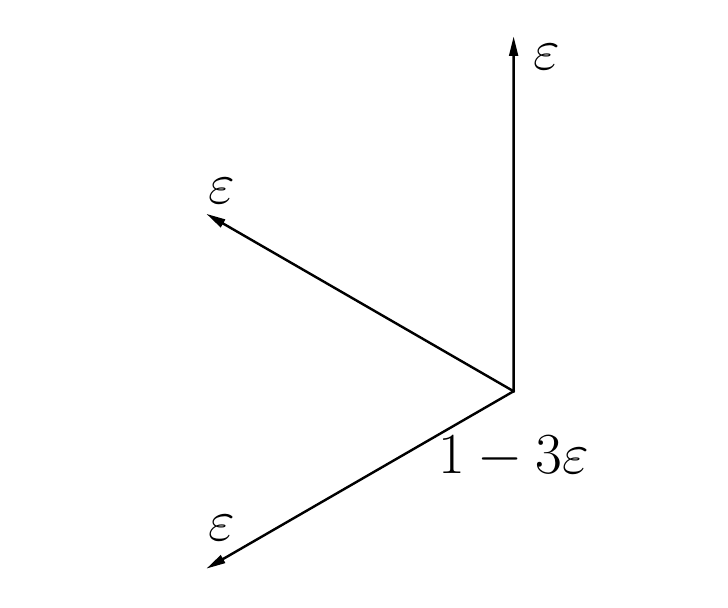} 
    \end{minipage}
    \caption{For $d = 2$, the algorithm sees as input the distribution on the left after the adversary corrupts $\eps$-mass. The three distributions to its right are ones consistent with the input.}
    \label{fig:robust_lower_bound}
\end{figure}

We first show the lower bound for $\eps \leq \frac{1}{d+1}$.  The hard instance is that the adversary hands over a distribution that
puts $\eps$ mass on each vertex of the regular simplex.  The true
distribution is the same, but with one of the vertices reflected across
the origin.  These distributions are all consistent with the
observed distribution -- that is, they have total variation at most $\varepsilon$ to the distribution handed to us by the adversary -- but have means at vertices of a simplex.  A
regular simplex is the setting where Jung's theorem is tight, and some
calculation gives a $JUNG_d \cdot \sqrt{2 \eps}$ lower bound.

When $d > \frac{1}{\eps} - 1$, we instead restrict to a lower-dimensional space of dimension $d' = \floor{\frac{1}{\eps} - 1}$ and apply the same bound to get a $JUNG_{d'}$ lower bound.  Since $d'$ is large, both $JUNG_d$ and $JUNG_{d'}$ are $\sqrt{2} - O(\eps)$.  
\section{Proof Details -- Heavy-Tailed Estimator}

Here, we provide a detailed description of our heavy-tailed estimator, along with key lemmas in the proof of our main result, Theorem~\ref{thm:heavy_upper}. We will focus on our $2$-dimensional estimator that achieves a constant better than $JUNG_2$; as stated earlier, we can ``lift'' it to high-dimensions to obtain a constant better than $JUNG_d$ in $d$ dimensions. We begin with the formal definition of ``inlier-light'' and ``outlier-light'' distributions. 

\subsection{``Inlier-Light'' and ``Outlier-Light'' Distributions}

\begin{restatable}[$(\beta, L)$-Inlier-Light Distribution]{definition}{inlierlightdef}
    A distribution $x$ over $\R$ with variance at most $\sigma^2$ is ``$(\beta, L)$-inlier-light'' if:
    \[
\E\left[(x - \mu)^2 \1_{|x - \mu| \le \beta T} \right] < (1- L) \sigma^2    \]
    for $T = \sigma \sqrt{\frac{n}{2 \log \frac{2}{\delta}}}$.
\end{restatable}

That is, a distribution is $(\beta, L)$-inlier-light if at most $(1 - L)$ fraction of its variance comes from ``inlier'' points, points within $\beta T$ of $\mu$. We define outlier-light analogously:

\begin{restatable}[$(\beta, L)$-Outlier-Light Distribution]{definition}{outlierlightdef}
    A distribution $x$ over $\R$ with variance at most $\sigma^2$ is ``$(\beta, L)$-outlier-light'' if:
    \[
\E\left[(x - \mu)^2 \1_{|x - \mu| \ge \beta T} \right] < L \sigma^2    \]
    for $T = \sigma \sqrt{\frac{n}{2 \log \frac{2}{\delta}}}$. 
    
    A distribution $x$ over $\R^d$ is $(\beta, L)$-outlier-light if $\inner{x, w}$ is $(\beta, L)$-outlier-light for all unit vectors $w$.
\end{restatable}

\subsection{Estimator for One-Dimensional Inlier-Light Distributions}
We first show that a variant of Catoni's Estimator for one-dimensional distributions, when computed using an appropriate $\psi$ function, achieves a rate strictly better than $\sigma \sqrt{\frac{2 \log \frac{1}{\delta}}{n}}$, the Gaussian rate, when the distribution is \emph{inlier-light}. \hyperref[alg:catoni_local]{\textsc{CatoniEstimatorLocal}} takes an initial estimate $\mu_0$ of the mean $\mu$ as input, such that $|\mu_0 - \mu| \le O\left(\sigma \sqrt{\frac{\log \frac{1}{\delta}}{n}} \right)$, typically computed using the median-of-means estimator~\citep{mom_estimator}.

\begin{algorithm}[H]
\caption{\textsc{CatoniEstimatorLocal}}
    \label{alg:catoni_local} 
    \vspace{0.2cm}
    \paragraph{Input parameters:}
    \begin{itemize}
        \item Failure probability $\delta$, One-dimensional iid samples $x_1, \dots, x_n$, Initial estimate $\mu_0$, $\psi$ function, Scaling parameter $T$.
    \end{itemize}
    \begin{enumerate}
        \item Compute
            \begin{align*}
                r(\mu_0) = \frac{T}{n} \sum_{i=1}^n \psi\left(\frac{x_i - \mu_0}{T} \right)
            \end{align*}
        \item Return mean estimate $\wh \mu = r(\mu_0) + \mu_0$
    \end{enumerate}
\end{algorithm}

We will suppose that our $\psi$ function satisfies the following.

\begin{restatable}{assumption}{inliercatonipsidef}
\label{assumption:improved_catoni_psi}
$\psi$ satisfies that for all $x$,
\begin{align*}
-\log\left(1 - x + \frac{x^2}{2}\right)\le \psi(x) \le \log\left(1 + x + \frac{x^2}{2} \right)
\end{align*}
Additionally, for constants $0 <\beta, \eta < 1$, for all $|x| \ge \frac{\beta}{2}$, 
\begin{align*}
    -\log\left( 1 - x + (1 - \eta) \frac{x^2}{2}\right)\le \psi(x) \le \log\left(1 + x + (1 - \eta) \frac{x^2}{2} \right)
\end{align*}
\end{restatable}
Recall that the ``$x - \frac{x^3}{6}$, Clipped'' function from Figure~\ref{fig:catoni_psi} satisfies that there exists an $\eta$ such that the above is satisfied for every $\beta$.
We show that for $(\beta, L)$-inlier-distributions, \hyperref[alg:catoni_local]{\textsc{CatoniEstimatorLocal}} improves upon the Gaussian rate by a $\approx \left( 1 - \frac{\eta L}{4} \right)$-factor when using a $\psi$ function satisfying the above, given an initial estimate $\mu_0$ of the mean.
\begin{restatable}[Improved Rate for One-Dimensional Inlier-Light Distributions]{lemma}{lemonedinlierlightcatoni}
    \label{lem:improved_catoni_local}
    For every constant $0 < \beta, L < 1$, $C_1 > 1$ there exists constant $C_2 > 1$ such that the following holds. Suppose $\psi$ satisfies Assumption~\ref{assumption:improved_catoni_psi}, $n > C_2 \log \frac{1}{\delta}$, and we have an initial estimate $\mu_0$ with $|\mu_0 - \mu| \le C_1 \sigma \sqrt{\frac{\log \frac{1}{\delta}}{n}}$.   
  We let $T = \sigma \sqrt{\frac{n}{2 \log \frac{2}{\delta}}}$. 
  
  Given $n$ one-dimensional iid samples $x_1, \dots, x_n$ with mean $\mu$ and variance at most $\sigma^2$, if $x_i$ is $(\beta, L)$-inlier-light, then, with probability $1 - \delta$, the output $\wh \mu$ of Algorithm~\hyperref[alg:catoni_local]{\textsc{CatoniEstimatorLocal}} satisfies
    \begin{align*}
        \left|\wh \mu - \mu \right| \le \left(1 - \frac{\eta L}{4} + C_2\frac{\log \frac{1}{\delta}}{n} \right) \cdot \sigma \cdot \sqrt{\frac{2 \log \frac{2}{\delta}}{n}}
    \end{align*}
\end{restatable}
\subsection{Testing Inlier-Light vs. Outlier-Light}
Our strategy will be to first test whether our two-dimensional samples come from a distribution that is \emph{inlier-light}, or \emph{outlier-light}, and use an appropriate estimator accordingly. Our tester (Algorithm~\hyperref[lem:2d_inlier_outlier_testing]{\textsc{2DInlierOutlierLightTesting}}, described in Appendix~\ref{app:inlier_outlier_test}) takes in $n$ samples along with initial estimates $\mu_0^{e_1}, \mu_0^{e_2}$ of the mean in directions $e_1, e_2$ respectively, and \emph{either} identifies a direction $e_j$ in which the distribution is inlier-light, \emph{or} certifies that the distribution is outlier-light in every direction. Formally,

\begin{restatable}[Two-dimensional Inlier-Light vs. Outlier-Light Test]{lemma}{lemtwodinlieroutliertesting}
    \label{lem:2d_inlier_outlier_testing}
    For every constant $\beta < \frac{1}{8}$, $L > 8 \beta$, and $C_1 > 1$, there exists constant $C_2 > 1$ such that the following holds. Suppose $n > C_2 \log \frac{1}{\delta}$ and suppose our initial estimates $\mu_0^{e_j}$ satisfy $|\mu_0^{e_j} - \inner{e_j, \mu}| \le C_1 \sigma \sqrt{\frac{\log \frac{1}{\delta}}{n}}$ for $j \in \{1, 2\}$. We let $T = \sigma \sqrt{\frac{n}{2 \log \frac{4}{\delta}}}$.

    Given $n$ two-dimensional iid samples $x_1, \dots, x_n$ with mean $\mu$ and covariance $\Sigma \preccurlyeq \sigma^2 I_2$, with probability $1 - \delta$, Algorithm~\hyperref[alg:2d_inlier_vs_outlier_light-tester]{\textsc{2DInlierOutlierLightTester}} satisfies the following.
    \begin{itemize}
        \item If the output is $e_j$, $\inner{e_j, x_i}$ is $(\beta, L)$-inlier-light
        \item If the output is $\perp$, $x_i$ is $\left(16 \beta, 16 L\right)$-outlier-light.  (That is, $\inner{x_i, w}$ is $\left(16 \beta, 16 L\right)$-outlier-light for all unit vectors $w$.)
    \end{itemize}
\end{restatable}

\subsection{Catoni-Based Estimator for Two-Dimensional Inlier-Light Distributions}
We recall the standard definition of a $\rho$-net of vectors over $\R^2$:

\begin{restatable}{assumption}{assumptionrhonetu}
    \label{assumption:rho_net_U}
    $U$ is a $\rho$-net of $O(1/\rho)$ unit vectors such that for every $v \in \S^1$, there exists a vector $u \in U$ with $\|v - u\| \le \rho$.
\end{restatable}

If our distribution over $\R^2$ is determined to be inlier-light in some direction $e_j$, we will make use of the following $2$-dimensional estimator.
\begin{algorithm}[H]\caption{\textsc{2DInlierLightEstimator}}
    \label{alg:two-d-catoni} 
    \vspace{0.2cm}
    \paragraph{Input parameters:}
    \begin{itemize}
        \item Failure probability $\delta$, Two-dimensional iid samples $x_1, \dots, x_n$, $\psi$ function, Scaling parameter $T$, Inlier-Outlier-Ligtness parameters $\beta, L$, Approximation parameters $0 < \xi, \tau < 1$, Set of unit vectors $U$, Initial estimates $\mu_0^u$ for $u \in U$. 
    \end{itemize}
    \begin{enumerate}
        \item For every $u \in U$, run Algorithm~\hyperref[alg:1d_inlier_vs_outlier_light_tester]{\textsc{1DInlierOutlierTester}} with samples $\inner{u, x_1}, \dots, \inner{u, x_n}$, Failure probability $\frac{\delta}{4|U|}$, initial estimate $\mu_0^u$, and Lightness parameters $\beta/32, L/32$. If the output is ``INLIER-LIGHT'', let $\alpha_u = 1 - \Theta(\tau)$. Otherwise, let $\alpha_u = 1 + \xi$.
        \item For every $u \in U$, run Algorithm~\hyperref[alg:catoni_local]{\textsc{CatoniEstimatorLocal}} with samples $\inner{u, x_1}, \dots, \inner{u, x_n}$, initial estimate $\mu_0^u$, and failure probability $\frac{\delta}{4|U|}$ and let the mean estimate obtained be $\wh \mu_u$. 
        \item For each $u \in U$, define set $S_u = \left\{w : |\inner{u, w} - \wh \mu_u| \le \alpha_u \cdot \sigma \sqrt{\frac{2 \log \frac{2}{\delta}}{n}}\right\}$. Let $S$ be the convex set given by $S := \cap_{u \in U} S_u$.
        \item Consider the minimum enclosing ball of set $S$ and return its center as the mean estimate $\wh \mu$.
    \end{enumerate}
\end{algorithm}
\hyperref[alg:two-d-catoni]{\textsc{2DInlierLightEstimator}} takes in a $\rho$-net $U$, in addition to the iid samples $x_1, \dots, x_n \in \R^2$ and failure probability $\delta$. For each net vector $u \in U$, it tests whether the distribution of the $\inner{u, x_i}$ is inlier-light, computes an estimate of the mean in direction $u$ using our $1$-d estimator for inlier-light distributions, and assigns a confidence interval accordingly. The final estimate $\wh \mu$ is the center of the minimum enclosing ball of the points that satisfy all $|U|$ confidence intervals. We show:

\begin{restatable}[Two-Dimensional Estimator for Inlier-Light Distributions]{lemma}{lemtwodinlierestimator}
    \label{lem:two-d-catoni}
    For every constant $0 <\beta<1/32, L > 32\beta$, and $C > 1$, there exist constants $\xi, \tau < 1$ such that the following holds. Suppose $n > O_\xi\left( \log \frac{1}{\delta}\right)$, and we have that $|\mu_0^u - \inner{u, \mu}| \le C \sigma \sqrt{\frac{\log \frac{1}{\delta}}{n}}$ for all $u \in U$. Suppose further that $\psi$ satisfies Assumption~\ref{assumption:improved_catoni_psi} for parameter $\beta/8$ and that $U$ satisfies Assumption~\ref{assumption:rho_net_U} for $\rho = \delta^{\Theta(\xi)}$. Let $T = \sigma \sqrt{\frac{n}{2 \log \frac{2}{\delta}}}$.

    Given $n$ two-dimensional iid samples $x_1, \dots, x_n$ with mean $\mu$ and covariance $\Sigma \preccurlyeq \sigma^2 I_2$ such that $\inner{e_k, x_i}$ is $(\beta, L)$-inlier-light, with probability $1 - \delta$, Algorithm~\hyperref[alg:two-d-catoni]{\textsc{2DInlierLightEstimator}} returns a mean estimate $\wh \mu$ with
    \begin{align*}
        \|\wh \mu - \mu\| \le (1 - \tau) \cdot JUNG_2 \cdot \sigma \sqrt{\frac{2 \log \frac{2}{\delta}}{n}}
    \end{align*}
\end{restatable}

\subsection{Trimmed-Mean-Based Estimator for Two-Dimensional Outlier-Light Distributions}
\begin{algorithm}[H]\caption{\textsc{2DOutlierLightEstimator}}
    \label{alg:two_dimensional_trimmed_mean_estimator} 
    \paragraph{Input parameters:}
    \begin{itemize}
        \item Failure probability $\delta$, Two-dimensional samples $x_1, \dots, x_n$, Initial estimates $\mu_0^{e_1}, \mu_0^{e_2}$, Scaling parameter $T$, Approximation parameters $0<\beta, \xi < 1$.
    \end{itemize}
    \begin{enumerate}
        \item Consider the subset of samples $X'$ obtained by throwing out any sample $x_i$ with $|\inner{e_j, x_i} - \mu_0^{e_j}| > \sqrt{\beta} T$ for either $e_1$ or $e_2$. Return estimate $\wh \mu = \frac{1}{n} \sum_{i \in X'} x_i$.
    \end{enumerate}
\end{algorithm}
For \emph{outlier-light} distributions, \hyperref[alg:two_dimensional_trimmed_mean_estimator]{\textsc{2DOutlierLightEstimator}} computes a simple trimmed-mean estimate, throwing out any  point more than $\sqrt{\beta} T$ away from the initial mean estimate in the $e_1, e_2$ directions.
\begin{restatable}[Two-Dimensional Estimator for Outlier-Light Distributions]{lemma}{lemtwodoutlierlightestimator}
    \label{lem:2d_trimmed_mean}
    Define $T =\sigma \sqrt{\frac{n}{2 \log \frac{2}{\delta}}}$. For any constant $\beta < 1$, let $x_1, \dots, x_n$ be iid samples from a two-dimensional $(\beta, O(\beta))$-outlier-light distribution with mean $\mu$ and covariance $\Sigma \preccurlyeq \sigma^2 I_2$. Then, the output of Algorithm~\hyperref[alg:two_dimensional_trimmed_mean_estimator]{\textsc{2DOutlierLightEstimator}} when given as input initial estimates $\mu_0^j$ satisfying $|\mu_0^j - \inner{e_j, \mu}| \le O\left(\sigma \sqrt{\frac{\log \frac{1}{\delta}}{n}} \right)$ outputs estimate $\wh \mu$ satisfying with probability $1 - \delta$,
    \begin{align*}
        \|\wh \mu - \mu \| \le \left(1 + O\left(\sqrt \beta\right) \right) \cdot OPT_1
    \end{align*}
\end{restatable}

\subsection{Final Two-Dimensional Estimator}
Finally, we put together the previous parts to obtain our final Algorithm~\hyperref[alg:final_two_dimensional_estimator]{\textsc{2DHeavyTailedEstimator}}.
\begin{algorithm}[H]
\caption{\textsc{2DHeavyTailedEstimator}}
    \label{alg:final_two_dimensional_estimator} 
    \vspace{0.2cm}
    \paragraph{Input parameters:}
    \begin{itemize}
        \item Failure probability $\delta$, Two-dimensional samples $x_1, \dots, x_n$, $\psi$ function, Scaling parameter $T$, Inlier-Outlier-Lightness parameters $\beta, L$, Approximation parameters $0< \xi, \tau < 1$, set of unit vectors $U$
    \end{itemize}
    \begin{enumerate}
        \item Using $\Theta(\xi) n$ samples, compute Median-of-Means estimates $\mu_0^{e_j}$ of the one-dimensional samples $\inner{e_j, x_i}$ with failure probability $\frac{\delta}{4(|U|+2)}$ for each $j \in \{1, 2\}$.
        \item Using $\Theta(\xi) n$ samples, compute Median-of-Means estimates $\mu_0^u$ of the one-dimensional samples $\inner{u, x_i}$ with failure probability $\frac{\delta}{4(|U|+2)}$ for each $u \in U$.
        \item Let the set of the remaining $(1 - \Theta(\xi)) n$ samples be $X'$. Run Algorithm~\hyperref[alg:2d_inlier_vs_outlier_light-tester]{\textsc{2DInlierOutlierLightTester}} using failure probability $\delta/4$, the samples in $X'$ and initial estimates $\mu_0^{e_1}, \mu_0^{e_2}$.
        \item If the output of \hyperref[alg:2d_inlier_vs_outlier_light-tester]{\textsc{2DInlierOutlierLightTester}} is some $e_j$, run \hyperref[alg:two-d-catoni]{\textsc{2DInlierLightEstimator}} using failure probability $\delta/8$, the samples in $X'$, and the initial estimates $\mu_0^u$, and output its mean estimate $\wh \mu$.
        \item If instead the output of \hyperref[alg:2d_inlier_vs_outlier_light-tester]{\textsc{2DInlierOutlierLightTester}} is $\perp$, run \hyperref[alg:two_dimensional_trimmed_mean_estimator]{\textsc{2DOutlierLightEstimator}} using failure probability $\delta/4$, the samples in $X'$ and initial estimates $\mu_0^{e_1}, \mu_0^{e_2}$. Return its output $\wh \mu$.
    \end{enumerate}
\end{algorithm}
\begin{restatable}[Final Two-Dimensional Estimator]{theorem}{thmtwodestimator}
    \label{thm:2d_estimator}
    For any sufficiently small constant $\tau > 0$, there exist constants $0 <\xi, \beta, L < 1$ such that the following holds. Suppose $n > O_\xi(\log \frac{1}{\delta})$ and $T = \sigma \sqrt{\frac{n}{2 \log \frac{2}{\delta}}}$. Suppose the set $U$ is a $\rho$-net, satisfying Assumption~\ref{assumption:rho_net_U} for $\rho = \delta^{\Theta(\xi)}$.

    Given $n$ two-dimensional samples $x_1, \dots, x_n$ with mean $\mu$ and covariance $\Sigma \preccurlyeq \sigma^2 I_2$, with probability $1 - \delta$, Algorithm~\hyperref[alg:final_two_dimensional_estimator]{\textsc{2DHeavyTailedEstimator}} returns an estimate $\wh \mu$ with
    \begin{align*}
        \|\wh \mu - \mu\| \le (1 - \tau) \cdot JUNG_2 \cdot \sigma \sqrt{\frac{2 \log \frac{2}{\delta}}{n}}
    \end{align*}
\end{restatable}

\begin{proof}
    First note that by classical results on Median-of-Means~\citep{mom_estimator} and a union bound, for every vector $v \in U \cup \{e_1, e_2\}$, we have with probability $1 - \delta/4$,
    \begin{align*}
        \left|\mu_0^v - \inner{v, \mu} \right| \le O\left(\sigma \sqrt{\frac{\log \frac{1}{\delta}}{n}} \right)
    \end{align*}
since $n > O_\xi( \log \frac{1}{\delta})$. For the remaining proof, we condition on the above. Now, by a union bound, there exist constants $0 < \beta, L < 1$, such that by Lemmas~\ref{lem:2d_inlier_outlier_testing},~\ref{lem:two-d-catoni}~and~\ref{lem:2d_trimmed_mean}, the following events happen with probability $1 - 3\delta/4$.
    \begin{itemize}
        \item If the output of Algorithm~\hyperref[alg:2d_inlier_vs_outlier_light-tester]{\textsc{2DInlierOutlierLightTester}} is $e_j$, $\inner{e_j, x_i}$ is $(\beta, L)$-inlier-light. On the other hand, if the output is $\perp$, $x_i$ is $(8 \beta, 8 L)$-outlier-light.
        \item If $\inner{e_j, x_i}$ is $(\beta, L)$-inlier-light, Algorithm~\hyperref[alg:two-d-catoni]{\textsc{2DInlierLightEstimator}} returns $\wh \mu$ with
        \begin{align*}
            \|\wh \mu - \mu\| &\le (1 - 2\tau + \Theta(\xi)) \cdot JUNG_2 \cdot \sigma \sqrt{\frac{2 \log \frac{2}{\delta}}{n}}\\
            &\le (1 - \tau)  \cdot JUNG_2 \cdot \sigma \sqrt{\frac{2 \log \frac{2}{\delta}}{n}}
        \end{align*}
        \item If $x_i$ is $(8\beta, 8L)$-outlier-light, for $L = O(\beta)$, Algorithm~\hyperref[alg:two_dimensional_trimmed_mean_estimator]{\textsc{2DOutlierLightEstimator}} returns $\wh \mu$ with
            \begin{align*}
                \| \wh \mu - \mu \| \le \left(1 + O\left(\sqrt{\beta} \right) \right) \cdot \sigma \sqrt{\frac{2 \log \frac{1}{\delta}}{n}}
            \end{align*}
    \end{itemize}
    So, with probability $1 - \delta$ in total, for $\beta$ small enough, 
    Algorithm~\hyperref[alg:final_two_dimensional_estimator]{\textsc{2DHeavyTailedEstimator}} returns estimate $\wh \mu$ with 
    \begin{align*}
        \| \wh \mu - \mu\| \le (1 - \tau) \cdot JUNG_2 \cdot \sigma \sqrt{\frac{2 \log \frac{2}{\delta}}{n}}
    \end{align*}
\end{proof}
\section{Open Questions}
Our work suggests a number of exciting avenues for future research. Some of these are:
\begin{itemize}
    \item What is the sharp rate for heavy tailed estimation when $\log \frac{1}{\delta} \gg d$? Our work establishes that it is not achieved by the naive strategy of aggregating one-dimensional estimates. Is it possible to achieve the Gaussian rate?
    \item Our upper and lower bounds are statistical---what about polynomial-time estimation? What are the sharp constants achievable, and is there a computational-statistical tradeoff? In particular,  for large $d$ no estimator achieving even the $JUNG_d \approx \sqrt{2}$ factor loss is known---current polynomial-time estimators~\citep{hopkins_subgaussian,pmlr-v99-cherapanamjeri19b} rely on the median-of-means framework, which loses constants even in one dimension.
    \item What are the sharp rates for other estimation problems under heavy-tailed noise? For instance, covariance estimation or regression?
\end{itemize}


\section*{Acknowledgements}
SBH was funded by NSF CAREER award no. 2238080 and MLA@CSAIL.  SG and
EP were funded by NSF awards CCF-2008868 and CCF-1751040 (CAREER).

\bibliographystyle{alpha}
\bibliography{references}
\appendix
\section{Vanilla One-Dimensional Catoni Estimator}\label{sec:catoni}
We first describe a variant of Catoni's one-dimensional estimator~\cite{catoni} for bounded variance distributions.

\setcounter{algorithm}{0}
\begin{algorithm}
\caption{\textsc{CatoniEstimatorLocal}}
    \vspace{0.2cm}
    \paragraph{Input parameters:}
    \begin{itemize}
        \item Failure probability $\delta$, One-dimensional iid samples $x_1, \dots, x_n$, Initial estimate $\mu_0$, $\psi$ function, Scaling parameter $T$.
    \end{itemize}
    \begin{enumerate}
        \item Compute
            \begin{align*}
                r(\mu_0) = \frac{T}{n} \sum_{i=1}^n \psi\left(\frac{x_i - \mu_0}{T} \right)
            \end{align*}
        \item Return mean estimate $\wh \mu = r(\mu_0) + \mu_0$
    \end{enumerate}
\end{algorithm}
\begin{assumption}
\label{assumption:vanilla_catoni_psi}
$\psi$ satisfies that for all $x$,
\begin{align*}
-\log\left(1 - x + \frac{x^2}{2}\right)\le \psi(x) \le \log\left(1 + x + \frac{x^2}{2} \right)
\end{align*}
\end{assumption}

\begin{lemma}
    \label{lem:vanilla_catoni_local}
    For every constant $C_1 > 1$ there exists constant $C_2 > 1$ such that the following holds. Suppose $\psi$ satisfies Assumption~\ref{assumption:vanilla_catoni_psi}, $n > C_2 \log \frac{1}{\delta}$, and we have an initial estimate $\mu_0$ with $|\mu_0 - \mu| \le C_1 \sigma \sqrt{\frac{\log \frac{1}{\delta}}{n}}$.   
  We let $T = \sigma \sqrt{\frac{n}{2 \log \frac{2}{\delta}}}$. Given $n$ one-dimensional iid samples $x_1, \dots, x_n$ with mean $\mu$ and variance at most $\sigma^2$, with probability $1 - \delta$, the output $\wh \mu$ of Algorithm~\hyperref[alg:catoni_local]{\textsc{CatoniEstimatorLocal}} satisfies
    \begin{align*}
        \left|\wh \mu - \mu \right| \le \left(1 + C_2 \frac{\log \frac{1}{\delta}}{n} \right) \cdot \sigma \cdot \sqrt{\frac{2 \log \frac{2}{\delta}}{n}}
    \end{align*}
\end{lemma}

\begin{proof}
    By Assumption~\ref{assumption:vanilla_catoni_psi}, we have
    \begin{align*}
        \E\left[\exp\left( \frac{n}{T} r(\mu_0) \right) \right] &= \prod_{i=1}^n \E\left[\exp\left(\psi\left(\frac{x_i - \mu_0}{T} \right)\right) \right]\\
        &\le \prod_{i=1}^n \E\left[\left( 1 + \frac{x_i - \mu_0}{T} + \frac{(x_i - \mu_0)^2}{2 T^2}\right)\right]\\
        &= \left(1 + \frac{\mu - \mu_0}{T} + \frac{1}{2T^2} \cdot \left[\sigma^2 + (\mu - \mu_0)^2 \right]  \right)^n\\
        &\le \exp\left(\frac{n}{T} (\mu - \mu_0) + \frac{n }{2T^2} \cdot \left[ \sigma^2 + (\mu - \mu_0)^2 \right] \right)
    \end{align*}
    So, by Markov's inequality, for $T = \sigma \sqrt{\frac{n}{2 \log \frac{2}{\delta}}}$, we have
    \begin{align*}
        &\Pr\left[r(\mu_0) \ge (\mu - \mu_0) + \sigma \cdot \sqrt{\frac{2 \log \frac{2}{\delta}}{n}}\left(1 + C_2 \frac{\log \frac{1}{\delta}}{n} \right) \right] \\
        &\le \exp\left(\frac{n}{2T^2} \cdot \left[\sigma^2 + (\mu - \mu_0)^2 \right] - \frac{n}{T} \cdot \sigma \cdot \sqrt{\frac{2 \log \frac{2}{\delta}}{n}}\left(1 + C_2 \frac{\log \frac{1}{\delta}}{n} \right) \right)\\
        &\le \exp\left(- \log \frac{2}{\delta}\right)\\
        &= \delta/2
    \end{align*}
    since $|\mu - \mu_{0}| \le C_1\sigma \sqrt{\frac{\log \frac{1}{\delta}}{n}} $, for $C_2 \ge \Omega(C_1^2)$. 

    Similarly, for the lower tail, the MGF is given by
    \begin{align*}
        \E\left[\exp\left(- \frac{n}{T} r(\theta_0) \right) \right] &\le \prod_{i=1}^n \E\left[\left(1 - \frac{x_i - \mu_0}{T} + \frac{(x_i - \mu_0)^2}{2T^2} \right) \right]\\
        &\le \exp\left( -\frac{n}{T}  (\mu - \mu_0) + \frac{n }{2T^2} \left[\sigma^2 + (\mu - \mu_0)^2 \right]\right)
    \end{align*}
    So, by Markov's inequality, for $T = \sigma \cdot \sqrt{\frac{n}{2 \log \frac{2}{\delta}}}$, we have
    \begin{align*}
        \Pr\left[r(\mu_0) \le (\mu - \mu_0) - \sigma \cdot \sqrt{\frac{2 \log \frac{2}{\delta}}{n}}\left(1 + C_2 \frac{\log \frac{1}{\delta}}{n} \right) \right] \le \delta/2
    \end{align*}
    Then, taking a union bound gives the claim.
\end{proof}
\setcounter{algorithm}{4}

\begin{algorithm}[H]\caption{\textsc{CatoniEstimator}}
    \label{alg:catoni_global} 
    \paragraph{Input parameters:}
    \begin{itemize}
        \item Failure probability $\delta$, One-dimensional iid samples $x_1, \dots, x_n$, $\psi$ function, Scaling parameter $T$, Approximation parameter $0 < \xi < 1$.
    \end{itemize}
    \begin{enumerate}
        \item Use the first $\Theta(\xi n)$ samples to compute the Median-of-Means estimate $\mu_0$ with failure probability $\delta/2$.
        \item Return the result $\wh \mu$ of Algorithm~\hyperref[alg:catoni_local]{\textsc{CatoniEstimatorLocal}} using initial estimate $\mu_0$, and the remaining $(1 - \Theta(\xi)) n$ samples, and failure probability $\delta/2$.
    \end{enumerate}
\end{algorithm}

\begin{theorem}
    \label{thm:1d_improved_catoni}
    For every constant $\xi > 0$, suppose $n > O_{\xi}(\log \frac{1}{\delta})$, $\psi$ satisfies Assumption~\ref{assumption:vanilla_catoni_psi}, and consider $T = \sigma \sqrt{\frac{n}{2 \log \frac{4}{\delta}}}$. Given $n$ one-dimensional iid samples $x_1, \dots, x_n$ with mean $\mu$ and variance at most $\sigma^2$, with probability $1 - \delta$, the output $\wh \mu$ of Algorithm~\hyperref[alg:catoni_global]{\textsc{CatoniEstimator}} satisfies
    \begin{align*}
        \left|\wh \mu - \mu \right| \le \left(1 + \xi \right) \cdot \sigma \cdot \sqrt{\frac{2 \log \frac{4}{\delta}}{n}}
    \end{align*}
\end{theorem}
\begin{proof}
    First, by classical results on Median-of-Means \cite{mom_estimator}, $\mu_0$ satisfies with probability $1 - \delta/2$,
    \begin{align*}
        |\mu_0 - \mu| \le C \sigma \sqrt{\frac{\log \frac{1}{\delta}}{n \xi}}
    \end{align*}
    for some constant $C > 0$. Then, invoking Lemma~\ref{alg:catoni_local} using the remaining $(1 - \Theta(\xi))n$ samples and failure probability $\delta/2$ gives the claim.
\end{proof}

\section{Improved Heavy-Tailed Estimator}
We will make use of the following notions of ``$(\beta, L)$-inlier-light'' and ``$(\beta, L)$-outlier-light'' distributions throughout this section.
\inlierlightdef*
\outlierlightdef*

\subsection{Improved One-Dimensional Catoni-Based Estimator when Inlier-Light}

The following assumption on $\psi$ functions is stronger than the Catoni requirement, and allows for a more accurate estimate when a distribution is $(\beta, L)$-inlier-light.
\inliercatonipsidef*

There exist $\psi$ functions such that for every $\beta > 0$, there exists $\eta > 0$ such that Assumption~\ref{assumption:improved_catoni_psi} is satisfied.
One such function is $\psi(x) = x - x^3/6$ for $\abs{x} \leq 1$ and $\psi(x) = \frac{5}{6} \sign(x)$ for $\abs{x} > 1$.
For $\abs{x} \lesssim 1$, there is $\Theta(x^4)$ flexibility in the choice of $\psi(x)$.

%
\lemonedinlierlightcatoni*
\begin{proof}
    By Assumption~\ref{assumption:improved_catoni_psi}, we have
   \begin{align*}
       &\E\left[\exp\left( \frac{n}{T} r(\mu_{0})\right) \right] \\
       &= \prod_{i=1}^n \E\left[\exp\left(\psi\left(\frac{x_i - \mu_0}{T} \right)\right) \right]\\
       &= \prod_{i=1}^n \left(\E\left[\exp\left(\psi\left(\frac{x_i - \mu_{0}}{T} \right)\right) \1_{|x_i - \mu_0| \le \beta T/2 } \right] + \E\left[\exp\left(\psi\left(\frac{x_i - \mu_{0}}{T} \right)\right) \1_{|x_i - \mu_0| > \beta T/2} \right] \right)\\
       &\le \prod_{i=1}^n \left(1 + \E\left[\frac{x_i - \mu_0}{T} \right] + \frac{1}{2 T^2}\left(\E\left[\left(x_i - \mu_0\right)^2 \1_{|x_i - \mu_0| \le \beta T/2}\right] + (1 - \eta) \E\left[ (x_i - \mu_0)^2 \1_{|x_i - \mu_0| > \beta T/2}\right]\right) \right)\\
   \end{align*} 
   Now, since $x_i$ is $(\beta, L)$-inlier-light, so that $\E\left[(x_i - \mu)^2 \1_{|x_i - \mu| \le \beta T} \right] < (1-L) \sigma^2$, we have
   \begin{align*}
        \E\left[(x_i - \mu)^2 \1_{|x_i - \mu_0| \le \beta T/2} \right] \le \E\left[(x_i - \mu)^2 \1_{|x_i - \mu| \le \beta T} \right] \le (1-L) \sigma^2
   \end{align*}
   since $|\mu - \mu_0| \le O\left(\sigma \sqrt{\frac{\log \frac{1}{\delta}}{n}} \right) \le \beta T/2$. So,
   \begin{align*}
       &\E\left[\left(x_i - \mu_0\right)^2 \1_{|x_i - \mu_0| \le \beta T/2}\right] + (1 - \eta) \E\left[ (x_i - \mu_0)^2 \1_{|x_i - \mu_0| > \beta T/2}\right] \\
       &\le (\mu - \mu_0)^2 + \E\left[(x_i - \mu)^2 \1_{|x_i - \mu_0| \le \beta T/2} \right] + (1-\eta)\E\left[(x_i - \mu)^2 \1_{|x_i - \mu_0| > \beta T/2} \right] \\
       &\quad- 2 \eta \E\left[(x_i - \mu)(\mu - \mu_0) \1_{|x_i - \mu_0| > \beta T/2} \right]\\
       &\le \left(1 - \frac{\eta L}{2}\right) \sigma^2
   \end{align*}
   since $n > C_2 \log \frac{1}{\delta}$.
   So,
   \begin{align*}
       \E\left[\exp\left(\frac{n}{T} r(\mu_0) \right) \right] &\le \prod_{i=1}^n \left(1 + \frac{\mu - \mu_0}{T} + \frac{\sigma^2}{2 T^2}\left(1 - \frac{\eta L}{2} \right) \right)\\
       &\le \exp\left(n \cdot \frac{\mu - \mu_0}{T} + \frac{n\sigma^2}{2 T^2} \left(1 - \frac{\eta L}{2} \right) \right)
   \end{align*}
   Then, by Markov's inequality, for $T = \sigma \sqrt{\frac{n}{2 \log \frac{2}{\delta}}}$,
   \begin{align*}
       &\Pr\left[r(\mu_0) \ge (\mu - \mu_{0}) + \left(1 - \frac{\eta L}{4} + \frac{C_2 \log \frac{1}{\delta}}{n} \right) \cdot \sigma \cdot \sqrt{\frac{2 \log \frac{2}{\delta}}{n}}\right]\\ &\le \exp\left(\frac{n \sigma^2}{2 T^2}\left(1 - \frac{\eta L}{2} \right)  - \frac{n}{T} \cdot \left(1 - \frac{\eta L}{4} + \frac{C_2 \log \frac{1}{\delta}}{n} \right) \cdot \sigma \cdot \sqrt{\frac{2 \log \frac{2}{\delta}}{n}} \right)\\
       &\le \exp\left(- \log \frac{2}{\delta}\right)\\
       &= \delta/2
   \end{align*}
   since $|\mu - \mu_0| \le C_1 \sigma \sqrt{\frac{2 \log \frac{1}{\delta}}{n}}$, for $C_2 \ge \Omega(C_1^2)$.
   
   Similarly, for the lower tail, the MGF is given by
   \begin{align*}
       \E\left[\exp\left( - \frac{n}{T} r(\mu_0) \right) \right] \le \exp\left(- n \cdot \frac{\mu - \mu_0}{T} + \frac{n\sigma^2}{T^2} \left(1 - \frac{\eta L}{2} \right)\right)
   \end{align*}
    so that by Markov's inequality,
    \begin{align*}
        \Pr\left[r(\mu_0) \le (\mu - \mu_{0}) - \left(1 - \frac{\eta L}{4} \right) \cdot \sigma \cdot \sqrt{\frac{2 \log \frac{2}{\delta}}{n}} \right] \le \delta/2
    \end{align*}
    for $T = \sigma \sqrt{\frac{n}{2 \log \frac{2}{\delta}}}$.
    
    Taking a union bound gives the claim.
\end{proof}

\subsection{Testing Inlier-Light vs. Outlier-Light}\label{app:inlier_outlier_test}
\begin{algorithm}[H]\caption{\textsc{1DInlierOutlierLightTester}}
    \label{alg:1d_inlier_vs_outlier_light_tester} 
    \paragraph{Input parameters:}
    \begin{itemize}
        \item Failure Probability $\delta$, One-dimensional iid samples $x_1, \dots, x_n$, Scaling parameter $T$, Inlier-Outlier-Lightness parameters $\beta, L$, Initial estimate $\mu_0$.
    \end{itemize}
    \begin{enumerate}
        \item Compute \[B = \frac{1}{n} \sum_{i=1}^n (x_i - \mu_0)^2 \1_{|x_i - \mu_0| \le 2 \beta T}\]
        \item If $B \le \left(1 - 2 L \right) \sigma^2$, return ``INLIER-LIGHT''. Otherwise return ``OUTLIER-LIGHT''.
    \end{enumerate}
\end{algorithm}

\begin{lemma}
    \label{lem:inlier_outlier_1d_testing}
    For every constant $\beta < 1/16, L > 16 \beta$ and $C_1 > 1$, there exists constant $C_2 > 1$ such that the following holds. Suppose $n > C_2 \log \frac{1}{\delta}$, and we have that $|\mu_0 - \mu| \le C_1 \sigma \sqrt{\frac{\log \frac{1}{\delta}}{n}}$. We let $T = \sigma \sqrt{\frac{n}{2 \log \frac{2}{\delta}}}$.
    
    Given $n$ one-dimensional iid samples $x_1, \dots, x_n$, with mean $\mu$ and variance at most $\sigma^2$, with probability $1-\delta$, we have that
    \begin{itemize}
        \item If Algorithm~\hyperref[alg:1d_inlier_vs_outlier_light_tester]{\textsc{1DInlierOutlierLightTester}} returns ``INLIER-LIGHT'', then $x_i$ is $(\beta, L)$-inlier-light, 
        \item If Algorithm~\hyperref[alg:1d_inlier_vs_outlier_light_tester]{\textsc{1DInlierOutlierLightTester}} returns ``OUTLIER-LIGHT'', then $x_i$ is $(4 \beta, 4 L)$-outlier-light
    \end{itemize}
\end{lemma}
\begin{proof}
    First, note that the variance of $(x_i - \mu)^2 \1_{|x_i - \mu| \le \beta T}$ is at most $(\beta T \sigma)^2$, and it is bounded by $(\beta T)^2$. Thus, by Bernstein's inequality, since $T = \sigma \sqrt{\frac{n}{2 \log \frac{2}{\delta}}}$ and $L > 8 \beta$, with probability $1 - \delta$,
    \begin{align*}
        \left|\frac{1}{n} \sum_{i=1}^n (x_i - \mu)^2 \1_{|x_i - \mu| \le \beta T} - \E\left[(x - \mu)^2 \1_{|x - \mu| \le \beta T} \right]\right| &\le \beta T \sigma \sqrt{\frac{2 \log \frac{2}{\delta}}{n}} + 2(\beta T)^2 \frac{\log \frac{2}{\delta}}{n} \\
        &= 2\beta \sigma^2 \le L\sigma^2/4
    \end{align*}
    and
    \begin{align*}
        \left|\frac{1}{n} \sum_{i=1}^n (x_i - \mu)^2 \1_{|x_i - \mu| \le 4\beta T} - \E\left[(x - \mu)^2 \1_{|x - \mu| \le 4\beta T} \right]\right| &\le 4\beta T \sigma \sqrt{\frac{2 \log \frac{2}{\delta}}{n}} + 8(\beta T)^2 \frac{\log \frac{2}{\delta}}{n} \\
        &= 8\beta \sigma^2 \le L\sigma^2
    \end{align*}
    We condition on the above. Now, since $|\mu_0 - \mu| \le C_1 \sigma \sqrt{\frac{\log \frac{1}{\delta}}{n}} \le \frac{\beta T}{2}$, we have
    \begin{align*}
        (x_i - \mu_0)^2 \1_{|x_i - \mu| \le \beta T} \le (x_i - \mu_0)^2 \1_{|x_i - \mu_0| \le 2 \beta T}
    \end{align*}
    So, since $|\mu - \mu_0| \le C_1 \sigma \sqrt{\frac{\log \frac{1}{\delta}}{n}}$, we have, by Cauchy-Schwarz,
    \begin{align*}
        &\frac{1}{n} \sum_{i=1}^n (x_i - \mu)^2 \1_{|x_i - \mu| \le \beta T}\\
        &\le \frac{1}{n}\sum_{i=1}^n (x_i - \mu_0)^2 \1_{|x_i - \mu_0| \le 2 \beta T} + C_1^2 \sigma^2 \frac{\log \frac{1}{\delta}}{n} + 2 C_1 \sigma \sqrt{\frac{\log \frac{1}{\delta}}{n}} \cdot \sqrt{\frac{1}{n}\sum_{i=1}^n (x_i - \mu_0)^2 \1_{|x_i - \mu_0| \le 2 \beta T } } \\
    \end{align*}
    So, if Algorithm~\hyperref[alg:1d_inlier_vs_outlier_light_tester]{\textsc{1DInlierOutlierLightTester}} returns ``INLIER-LIGHT'' so that $\frac{1}{n} \sum_{i=1}^n (x_i - \mu_0)^2 \1_{|x_i - \mu_0| \le 2 \beta T} \le \left(1 - 2 L \right) \sigma^2$, then,
    \begin{align*}
        \frac{1}{n} \sum_{i=1}^n (x_i - \mu)^2 \1_{|x_i - \mu| \le \beta T} &\le \left(1 - \frac{3 L}{2}\right) \sigma^2
    \end{align*}
    for $C_2$ large enough. So, in this case
    \begin{align*}
        \E\left[(x- \mu)^2 \1_{|x- \mu| \le \beta T} \right] \le \left(1 - L \right) \sigma^2
    \end{align*}
    so that $x_i$ is $(\beta, L)$-inlier-light, as claimed. For the other case, note that
    \begin{align*}
        (x_i - \mu_0)^2\1_{|x_i - \mu|\le 4 \beta T} \ge (x_i - \mu_0)^2 \1_{|x_i - \mu_0| \le 2 \beta T}
    \end{align*}
    Again, by Cauchy-Schwarz, since $|\mu - \mu_0| \le C_1 \sigma \sqrt{\frac{\log \frac{1}{\delta}}{n}}$,
    \begin{align*}
        &\frac{1}{n} \sum_{i=1}^n (x_i - \mu)^2 \1_{|x_i - \mu| \le 4 \beta T}\\
        &\ge \frac{1}{n} \sum_{i=1}^n (x_i - \mu_0)^2 \1_{|x_i - \mu_0| \le 2 \beta T} - C_1^2 \sigma^2 \frac{\log \frac{1}{\delta}}{n} - 2 C_1 \sigma \sqrt{\frac{\log \frac{1}{\delta}}{n}} \cdot \sqrt{\frac{1}{n} \sum_{i=1}^n (x_i - \mu_0)^2 \1_{|x_i - \mu_0| \le 2 \beta T}}
    \end{align*}
    So, when Algorithm~\hyperref[alg:1d_inlier_vs_outlier_light_tester]{\textsc{1DInlierOutlierLightTester}} returns ``OUTLIER-LIGHT'' so that $\frac{1}{n} \sum_{i=1}^n (x_i - \mu_0)^2 \1_{|x_i - \mu_0| \le 2 \beta T} > (1 - 2 L) \sigma^2$,
    \begin{align*}
        \frac{1}{n} \sum_{i=1}^n (x_i - \mu)^2 \1_{|x_i - \mu| \le 4 \beta T} > \left( 1 - 3 L\right)\sigma^2
    \end{align*}
    for $C_2$ large enough. So,
    \begin{align*}
        \E\left[(x - \mu)^2 \1_{|x - \mu| \le 4 \beta T} \right] > (1 - 4 L) \sigma^2
    \end{align*}
    So, since the variance is at most $\sigma^2$,
    \begin{align*}
        \E\left[(x-\mu)^2 \1_{|x - \mu| > 4 \beta T} \right] \le 4 L \sigma^2
    \end{align*}
    so that $x_i$ is $(4\beta, 4L)$-outlier-light, as claimed.
\end{proof}

\begin{lemma}
    \label{lem:1d_if_inlier_light_then_detects}
    For every constant $\beta < 1/16$, $L > 16 \beta$, and $C_1 > 1$, there exists constant $C_2 > 1$ such that the following holds. Suppose $n > C_2 \log \frac{1}{\delta}$, and we have that $|\mu_0 - \mu| \le C_1 \sigma \sqrt{\frac{\log \frac{1}{\delta}}{n}}$. We let $T = \sigma \sqrt{\frac{n}{2 \log \frac{2}{\delta}}}$.

    Given $n$ one-dimensional iid samples $x_1, \dots, x_n$ with mean $\mu$ and variance at most $\sigma^2$, with probability $1 - \delta$, we have that
    \begin{itemize}
        \item If $x_i$ is $(4\beta, 4L)$-inlier-light, then Algorithm~\hyperref[alg:1d_inlier_vs_outlier_light_tester]{\textsc{1DInlierOutlierLightTester}} returns ``INLIER-LIGHT''.
    \end{itemize}
\end{lemma}
\begin{proof}
    The proof is similar to the proof of Lemma~\ref{lem:inlier_outlier_1d_testing}.
    First, note that the variance of $(x_i - \mu)^2 \1_{|x_i - \mu| \le 4\beta T}$ is at most $(4\beta T \sigma)^2$, and it is bounded by $(4\beta T)^2$. Thus, by Bernstein's inequality, since $T = \sigma \sqrt{\frac{n}{2 \log \frac{2}{\delta}}}$ and $L > 16 \beta$, with probability $1 - \delta$,
    \begin{align*}
        \left|\frac{1}{n} \sum_{i=1}^n (x_i - \mu)^2 \1_{|x_i - \mu| \le 4\beta T} - \E\left[(x - \mu)^2 \1_{|x - \mu| \le 4\beta T} \right]\right| &\le 4\beta T \sigma \sqrt{\frac{2 \log \frac{2}{\delta}}{n}} + 2(4\beta T)^2 \frac{\log \frac{2}{\delta}}{n} \\
        &\le 4\beta \sigma^2 \le L\sigma^2/4
    \end{align*}
    We condition on the above. Now, since $|\mu_0 - \mu| \le C_1 \sigma \sqrt{\frac{\log \frac{1}{\delta}}{n}} \le \frac{\beta T}{2}$, we have
    \begin{align*}
        (x_i - \mu_0)^2 \1_{|x_i - \mu_0| \le 2 \beta T} \le (x_i - \mu_0)^2 \1_{|x_i - \mu| \le 4 \beta T}
    \end{align*}
    So, since $|\mu - \mu_0| \le C_1 \sigma \sqrt{\frac{\log \frac{1}{\delta}}{n}}$, we have, by Cauchy-Schwarz,
    \begin{align*}
        &\frac{1}{n} \sum_{i=1}^n (x_i - \mu_0)^2 \1_{|x_i - \mu| \le 4\beta T}\\
        &\le \frac{1}{n}\sum_{i=1}^n (x_i - \mu)^2 \1_{|x_i - \mu| \le 4 \beta T} + C_1^2 \sigma^2 \frac{\log \frac{1}{\delta}}{n} + 2 C_1 \sigma \sqrt{\frac{\log \frac{1}{\delta}}{n}} \cdot \sqrt{\frac{1}{n}\sum_{i=1}^n (x_i - \mu)^2 \1_{|x_i - \mu| \le 4 \beta T } } \\
    \end{align*}
    So, if $x_i$ is $(4 \beta, 4L)$-inlier-light so that by the above
    \[
        \frac{1}{n} \sum_{i=1}^n (x_i - \mu)^2 \1_{|x_i - \mu| \le 4 \beta T} \le \left( 1 - 3 L\right) \sigma^2
    \]
    we have
    \[
        \frac{1}{n} \sum_{i=1}^n (x_i - \mu_0)^2 \1_{|x_i - \mu_0| \le 2 \beta T} \le (1 - 2 L) \sigma^2
    \]
    so that Algorithm~\hyperref[alg:1d_inlier_vs_outlier_light_tester]{\textsc{1DInlierOutlierLightTester}} returns ``INLIER-LIGHT'' as claimed.
\end{proof}
\begin{algorithm}[H]\caption{\textsc{2DInlierOutlierLightTester}}
    \label{alg:2d_inlier_vs_outlier_light-tester} 
    \paragraph{Input parameters:}
    \begin{itemize}
        \item Two-dimensional iid samples $x_1, \dots, x_n$, Scaling parameter $T$, Inlier-Outlier-Lightness parameters $\beta, L$, Initial estimates $\mu_0^{e_1}, \mu_0^{e_2}$.
    \end{itemize}
    \begin{enumerate}
        \item For each $j \in \{1, 2\}$, run Algorithm~\hyperref[alg:1d_inlier_vs_outlier_light_tester]{\textsc{1DInlierOutlierLightTester}} using samples $\inner{e_j, x_1}, \dots, \inner{e_j, x_n}$ and initial estimate $\mu_0^{e_j}$.
        \item If for both $j$ the output is ``OUTLER-LIGHT'', return $\perp$. Otherwise, return $e_j$ such that the output for run $j$ was ``INLIER-LIGHT''.
    \end{enumerate}
\end{algorithm}

\begin{lemma}
    \label{lem:1_direction_to_all_outlier_lightness}
    Suppose $x$ is a distribution over $\R^2$ with mean $\mu$ and covariance $\Sigma \preccurlyeq \sigma^2 I_d$ such that $\inner{v_j, x}$ is $(\beta, L)$-outlier-light for each $j \in [2]$. Suppose also that $|\inner{v_1, v_2}| \le 3/4$.  Then $x$ is $(4 \beta, 4 L)$-outlier-light.
\end{lemma}
\begin{proof}  
    For any $w \in \S^1$, there is some $j \in \{1, 2\}$ with $\inner{w, e_j} \ge \frac{1}{2}$.
    So,
    \begin{align*}
        \E\left[\inner{w, x - \mu}^2 \1_{|\inner{w, x - \mu}| > 4 \beta T} \right] \le 4 \E\left[\inner{e_j, x - \mu}^2 \1_{|\inner{e_j, x -\mu}| > \beta T} \right] \le 4 L\sigma^2
    \end{align*}
    as required.
\end{proof}

\lemtwodinlieroutliertesting*
%
\begin{proof}
    By Lemma~\ref{lem:inlier_outlier_1d_testing}, with probability $1 - 2\delta$,
    \begin{itemize}
        \item If the output is $e_j$, then $\inner{e_j, x_i}$ is $(\beta, L)$-inlier-light as claimed.
        \item If the output is $\perp$, then both $\inner{e_1, x_i}$ and $\inner{e_2, x_i}$ are $(4 \beta, 4L)$-outlier-light. Then, by Lemma~\ref{lem:1_direction_to_all_outlier_lightness}, $x_i$ is $(16 \beta, 16 L)$-outlier-light.
    \end{itemize}
    Reparameterizing $\delta$ gives the claim.
\end{proof}

\subsection{Properties of Inlier-Lightness}
\begin{lemma}
    \label{lem:one_direction_inlier_light_implies_two_triangle_directions_inlier_light}
    Suppose $x$ is a two-dimensional distribution with mean $\mu$ and covariance $\Sigma \preccurlyeq \sigma^2 I_d$ such that $\inner{e_i, x}$ is $(\beta, L)$-inlier-light. Consider vectors $v_1, v_2, v_3$ such that for each $j \neq k$, $|\inner{v_j, v_k}| \le \frac{3}{4}$. Then, for some $j$, $\inner{v_j, x}$ is $(\beta/8, L/8)$-inlier-light.
\end{lemma}
\begin{proof}
    Under the constraints provided, there exists two $j \in [3]$ with $|\inner{v_j, e_i}| \ge \frac{1}{2}$. Then, there are two cases
    \begin{itemize}
        \item \textbf{$\Var(\inner{e_i, x}) \le \left(1 - \frac{L}{2} \right)\sigma^2$.} In this case, 
        \begin{align*}
            \Var(\inner{v_j, x}) &\le \frac{3}{4} \sigma^2 + \frac{1}{2} \left(1 - \frac{L}{2} \right) \sigma^2\\
            &\le \left(1 - \frac{L}{4} \right)\sigma^2
        \end{align*}
        so that $\inner{v_j, x}$ is $(\beta/4, L/4)$-inlier-light.
        \item \textbf{$\Var(\inner{e_i, x}) > \left(1 - \frac{L}{2} \right)\sigma^2$.}  In this case, since $\inner{e_i, x}$ is $(\beta, L)$-inlier-light,
        \begin{align*}
            \E\left[(\inner{e_i, x - \mu})^2 \1_{|\inner{e_i, x-\mu} > \beta T|} \right] &\ge \left(1 - \frac{L}{2}  \right)\sigma^2 - \E\left[(\inner{e_i, x - \mu})^2 \1_{|\inner{e_i, x-\mu} \le \beta T|} \right]\\
            &\ge \frac{L}{2} \sigma^2
        \end{align*}
        so that $\inner{e_i, x}$ is $(\beta/2, L/2)$-outlier-heavy. Then, by (the contrapositive of) Lemma~\ref{lem:1_direction_to_all_outlier_lightness}, one of the two $\inner{v_j, x}$ is $(\beta/8, L/8)$ outlier-heavy, and hence $(\beta/8, L/8)$-inlier-light.
    \end{itemize}
\end{proof}
%
\setcounter{algorithm}{1}
\subsection{Two-Dimensional Catoni-Based Estimator when Inlier-Light}
\begin{algorithm}[H]\caption{\textsc{2DInlierLightEstimator}}
    \vspace{0.2cm}
    \paragraph{Input parameters:}
    \begin{itemize}
        \item Failure probability $\delta$, Two-dimensional iid samples $x_1, \dots, x_n$, $\psi$ function, Scaling parameter $T$, Inlier-Outlier-Ligtness parameters $\beta, L$, Approximation parameters $0 < \xi, \tau < 1$, Set of unit vectors $U$, Initial estimates $\mu_0^u$ for $u \in U$. 
    \end{itemize}
    \begin{enumerate}
        \item For every $u \in U$, run Algorithm~\hyperref[alg:1d_inlier_vs_outlier_light_tester]{\textsc{1DInlierOutlierLightTester}} with samples $\inner{u, x_1}, \dots, \inner{u, x_n}$, Failure probability $\frac{\delta}{4|U|}$, initial estimate $\mu_0^u$, and Lightness parameters $\beta/32, L/32$. If the output is ``INLIER-LIGHT'', let $\alpha_u = 1 - \Theta(\tau)$. Otherwise, let $\alpha_u = 1 + \xi$.
        \item For every $u \in U$, run Algorithm~\hyperref[alg:catoni_local]{\textsc{CatoniEstimatorLocal}} with samples $\inner{u, x_1}, \dots, \inner{u, x_n}$, initial estimate $\mu_0^u$, and failure probability $\frac{\delta}{4|U|}$ and let the mean estimate obtained be $\wh \mu_u$. 
        \item For each $u \in U$, define set $S_u = \left\{w : |\inner{u, w} - \wh \mu_u| \le \alpha_u \cdot \sigma \sqrt{\frac{2 \log \frac{2}{\delta}}{n}}\right\}$. Let $S$ be the convex set given by $S := \cap_{u \in U} S_u$.
        \item Consider the minimum enclosing ball of set $S$ and return its center as the mean estimate $\wh \mu$.
    \end{enumerate}
\end{algorithm}

\assumptionrhonetu*

This assumption is satisfied by a standard $\rho$-net in two-dimensions. Then, we have the main result of this section - that if our distribution is inlier-light in some direction $e_j$, then, Algorithm~\ref{alg:two-d-catoni} outputs an estimate that has error smaller than $JUNG_2 \cdot \sigma \sqrt{\frac{2 \log \frac{2}{\delta}}{n}}$ by a constant factor.

%

%
\lemtwodinlierestimator*
%
\begin{proof}
    First, by Lemma~\ref{lem:1d_if_inlier_light_then_detects} and a union bound, with probability $1 - \frac{\delta}{4}$, for any $u \in U$ such that $\inner{u, x_i}$ is $(\beta/8, L/8)$-inlier-light, Algorithm~\hyperref[alg:1d_inlier_vs_outlier_light_tester]{\textsc{1DInlierOutlierLightTester}} returns ``INLIER-LIGHT'', so that $\alpha_u = 1 - \Theta(\tau)$ for all such $u$.
    
    By Theorem~\ref{alg:catoni_local} and the union bound, with probability $1 - \frac{\delta}{4}$, since $|U| = O(1/\rho) = \delta^{-\Theta(\xi)}$, we have that for every $u \in U$,
    \begin{align*}
        \left|\wh \mu_u - \inner{u, \mu} \right| &\le \left(1 + O(\xi) \right) \cdot \sigma \cdot \sqrt{\frac{2 \log \frac{1}{\rho} + 2 \log \frac{2}{\delta}}{n}}\\
        &\le (1 + \xi) \cdot \sigma \cdot \sqrt{\frac{2 \log \frac{2}{\delta}}{n}}
    \end{align*}
    Similarly, for every $u$ such that $\inner{u, x_i}$ is $(\beta/8, L/8)$-inner-light, by Theorem~\ref{thm:1d_improved_catoni} and a union bound, with probability $1 - \frac{\delta}{4}$,
    \begin{align*}
        \left|\wh \mu_{u} - \inner{u, \mu} \right| \le \left(1 - \frac{\eta L}{4} + \xi \right) \cdot \sigma \sqrt{\frac{2 \log \frac{2}{\delta}}{n}}
    \end{align*}
    So, for constant $\xi$ sufficiently small, there is a constant $\tau > 0$ with 
    \begin{align*}
        \left|\wh \mu_{u} - \inner{u, \mu} \right| \le \left(1 - \Theta(\tau)\right) \cdot \sigma \sqrt{\frac{2 \log \frac{2}{\delta}}{n}}
    \end{align*}
    With probability $1 - \delta$, all the above conditions hold, so that for any $u \in U$ that has $\inner{u, x_i}$ that is $(\beta/8, L/8)$-inner-light, we have that $\wh \mu_u$ has smaller error than in the general case, and $\alpha_u$ captures this error. We condition on this event. 

    Then, if $R$ is the circumradius of the set $S$ in Algorithm~\hyperref[alg:two-d-catoni]{\textsc{2DInlierLightEstimator}}, its center $\wh \mu$ satisfies for every $u \in U$,
    \begin{align*}
        |\inner{u, \wh \mu - \mu}| \le R
    \end{align*}
    since by definition, the true mean $\mu$ lies in $S$.
    So,
    \begin{align*}
        \|\wh \mu - \mu\| = \sup_{w : \|w\| = 1} \inner{w, \wh \mu - \mu} \le \sup_{v \in U \cup \{e_j\}} \inner{v, \wh \mu - \mu} + \rho \|\wh \mu - \mu\| \le R + \rho\| \wh \mu - \mu\|
    \end{align*}
    so that
    \begin{align*}
        \|\wh \mu - \mu\| \le (1 + O(\rho)) R = (1 + \xi) R
    \end{align*}
    since $\rho = \delta^{\Theta(\xi)} \le \xi$.
    So, it suffices to bound $R$ by $(1 - \Theta(\tau)) \cdot JUNG_2 \cdot \sigma \sqrt{\frac{2 \log \frac{2}{\delta}}{n}}$, since for $\tau$ a small enough constant, this would imply
    \begin{align*}
        \|\wh \mu - \mu\| \le (1 - \tau) \cdot JUNG_2 \cdot \sigma \sqrt{\frac{2 \log \frac{2}{\delta}}{n}} 
    \end{align*}
    as required.
    
    To do this, note that if we consider the set $S$ along with its circumcircle, since $S$ is convex, there must be a triangle contained in $S$ whose vertices touch the circumcircle. Let $v_1, v_2, v_3$ be the unit vectors aligned with the sides of this triangle. There are two cases:
    \begin{itemize}
        \item \textbf{There exists a pair $i \neq j$ such that $|\inner{v_i, v_j}| > \frac{3}{4}$.} In this case, $R$ must be small. In particular, since the diameter of $S$ is at most $2 \cdot (1 + \xi) \cdot \sigma \sqrt{\frac{2 \log \frac{2}{\delta}}{n}}$ each side length corresponding to $v_i, v_j$, say $a_i, a_j$ must be at most this quantity. But by law of cosines, the other side must have length at most $\frac{2}{\sqrt{2}} (1 + \xi) \cdot \sigma \sqrt{\frac{2 \log \frac{2}{\delta}}{n}}$. But for a triangle with sides $a, b, c$, the circumradius is equal to $\frac{a b c}{\sqrt{(a+b +c)(b+c-a)(c+a-b)(a+b-c)}}$, which is monotonic in $a, b, c$. So, we have
        \begin{align*}
            R \le \frac{2 \sqrt{2}}{\sqrt{7}} \cdot (1 + \xi) \cdot \sigma \sqrt{\frac{2 \log \frac{2}{\delta}}{n}} \le (1 - \Theta(\tau)) \cdot \sqrt{\frac{4}{3}} \cdot \sigma \sqrt{\frac{2 \log \frac{2}{\delta}}{n}} = (1 - \Theta(\tau)) \cdot JUNG_2 \cdot \sigma \sqrt{\frac{2 \log \frac{2}{\delta}}{n}}
        \end{align*}
        as required.
        \item \textbf{For every pair $i \neq j$, $|\inner{v_i, v_j}| \le \frac{3}{4}$.} Then, since $\inner{e_k, x_i}$ is $(\beta, L)$-inlier-light, by Lemma~\ref{lem:one_direction_inlier_light_implies_two_triangle_directions_inlier_light}, there exists an $l \in [3]$ such that $v_l$ is $(\beta/8, L/8)$-inlier-light. Then, by the above, we have that $\alpha_{v_l} = 1 - \Theta(\tau)$ so that the side of the triangle corresponding to $v_l$ has length at most $(1 - \Theta(\tau)) \cdot \sigma \sqrt{\frac{2 \log \frac{2}{\delta}}{n}}$. But this means that $R$, the circumradius of a triangle with all two side lengths bounded by $(1 + \xi) \cdot \sigma \sqrt{\frac{2 \log \frac{2}{\delta}}{n}}$, and the third bounded by $(1 - \Theta(\tau)) \cdot \sigma \sqrt{\frac{2 \log \frac{2}{\delta}}{n}}$ has
        \begin{align*}
            R \le (1 - \Theta(\tau)) \cdot \sigma \sqrt{\frac{2 \log \frac{2}{\delta}}{n}}
        \end{align*}
        as required.
    \end{itemize}
    
\end{proof}

\subsection{Two-Dimensional Trimmed Mean Estimator when Outlier-Light}
\begin{algorithm}[H]\caption{\textsc{2DOutlierLightEstimator}}
    \paragraph{Input parameters:}
    \begin{itemize}
        \item Failure probability $\delta$, Two-dimensional samples $x_1, \dots, x_n$, Initial estimates $\mu_0^{e_1}, \mu_0^{e_2}$, Scaling parameter $T$, Approximation parameters $0<\beta, \xi < 1$.
    \end{itemize}
    \begin{enumerate}
        \item Consider the subset of samples $X'$ obtained by throwing out any sample $x_i$ with $|\inner{e_j, x_i} - \mu_0^{e_j}| > \sqrt{\beta} T$ for either $e_1$ or $e_2$. Return estimate $\wh \mu = \frac{1}{n} \sum_{i \in X'} x_i$.
    \end{enumerate}
\end{algorithm}
\begin{lemma}
\label{lem:trimming_doesnot_change_mean_much}
For any constants $\beta, L < 1$, suppose $x$ is a $(\beta, L)$-outlier-light distribution with mean $\mu$ and variance at most $\sigma^2$. Then, for $T = \sigma \sqrt{\frac{n}{2 \log \frac{2}{\delta}}}$, we have the following.
    \begin{align*}
        \left|\E\left[x \1_{|x - \mu| \le 2 \sqrt{\beta} T} \right] - \mu \right| \lesssim \frac{L}{\sqrt{\beta}} \cdot OPT_1
    \end{align*}
\end{lemma}

\begin{proof}
     We have
     \begin{align*}
         \left|\E\left[x \1_{|x - \mu| \le 2 \sqrt{\beta} T} \right] - \mu\right| &\le \left|\E\left[x \1_{|x - \mu| > 2 \sqrt{\beta} T} \right]\right|\\
         &\le \E\left[|x| \1_{|x - \mu| > 2 \sqrt{\beta} T} \right]\\
         &= \int_{2 \sqrt \beta T}^\infty \Pr\left[|x - \mu| \ge t \right] dt
     \end{align*}
     Note that $\Pr\left[|x - \mu| \ge t \right] \lesssim \frac{ L \sigma^2}{t^2}$ for $t > 2 \sqrt{\beta} T$ since $x$ is $(\beta, L)$-outlier-light, so that 
     \[
         \E\left[(x - \mu)^2 \1_{|x - \mu| > 2\sqrt{\beta} T} \right] \le \E\left[(x-\mu)^2 \1_{|x-\mu| > \beta T} \right] < L \sigma^2
     \]
     So,
     \begin{align*}
         |\E\left[x \1_{|x - \mu| \le 2 \sqrt{\beta} T} \right] - \mu| &\le \int_{2 \sqrt{\beta} T}^\infty \frac{L \sigma^2}{t^2} dt\\
         &\lesssim \frac{L \sigma^2}{\sqrt{\beta} T}\\
         &\lesssim \frac{L}{\sqrt{\beta}} \cdot \sigma \sqrt{\frac{2 \log \frac{1}{\delta}}{n}}
     \end{align*}
\end{proof}

\begin{lemma}
    \label{lem:1d_trimmed_mean_robustness}
     Define $T = \sigma \sqrt{\frac{n}{2 \log \frac{2}{\delta}}}$. Let $x$ be a one-dimensional distribution supported in $[-2 \sqrt \beta T, 2 \sqrt \beta T]$. Let $w \in \{0, 1\}$ be jointly distributed with $x$ such that:
    \begin{itemize}
        \item $\Pr[w = 0] \le O\left(\frac{\log \frac{1}{\delta}}{n} \right)$
    \end{itemize}
    Then,
    \begin{align*}
        |\E[w x] - \E[x]| \lesssim \beta \cdot OPT_1
    \end{align*}
\end{lemma}

\begin{proof}
    Since $w \in \{0, 1\}$,
    \begin{align*}
        \left|\E[w x] - \E[x]\right| = |\E[x\1_{w = 0}]| \le 2 \sqrt \beta T \cdot O\left(\frac{\log \frac{1}{\delta}}{n} \right) \lesssim \sqrt \beta \sigma \sqrt{\frac{2 \log \frac{1}{\delta}}{n}}
    \end{align*}
\end{proof}

\begin{lemma}
    \label{lem:trimmed_mean_1d}
   Define $T = \sigma \sqrt{\frac{n}{2 \log \frac{2}{\delta}}}$.
  Let $x$ be a one-dimensional $(\beta, L)$-outlier-light distribution with mean $\mu$ and variance at most $\sigma^2$.
    Let $w \in \{0, 1\}$ be jointly distributed with $x$ such that:
    \begin{itemize}
        \item $\Pr[w = 0] \le O\left(\frac{\log \frac{1}{\delta}}{n} \right)$
        \item $w = 0$ if $\abs{x - \mu} > 2 \sqrt{\beta} T$
    \end{itemize}
    Then with $1-\delta$ probability, given $n$ independent samples $(x_i, w_i)$, we have
    \[
    \left|\frac{1}{n} \sum_{i=1}^n w_i x_i - \mu\right| \leq \left(1 + O\left(\sqrt{\beta} + \frac{L}{\sqrt{\beta}}\right)\right) \cdot OPT_1.
    \]
\end{lemma}
\begin{proof}
    First, note that
    \begin{align*}
        \left|\E\left[w x \right] - \mu\right| &= \left|\E\left[wx \1_{|x - \mu| \le 2 \sqrt \beta T} \right] - \mu\right|\\
        &\le \left|\E\left[wx \1_{|x - \mu| \le 2 \sqrt \beta T} \right] - \E\left[ x \1_{|x - \mu| \le 2 \sqrt \beta T}\right]\right| + \left|\E\left[ x \1_{|x - \mu| \le 2 \sqrt \beta T}\right] - \mu\right|\\
        &\lesssim \left(\sqrt \beta + \frac{L}{\sqrt \beta} \right)\cdot OPT_1
    \end{align*}
    by Lemma~\ref{lem:trimming_doesnot_change_mean_much}~and~\ref{lem:1d_trimmed_mean_robustness}.

    Now since $|wx - \mu| \le 2 \sqrt \beta T$, and its variance is at most $\sigma^2$, by Bernstein's inequality,
    \begin{align*}
        \left|\frac{1}{n} \sum_{i=1}^n w_i x_i - \E\left[w x \right] \right| &\le \sigma \sqrt{\frac{2 \log \frac{1}{\delta}}{n}} + O\left(\sqrt \beta T \cdot \frac{\log \frac{1}{\delta}}{n} \right)\\
        &\le \left(1 + O(\sqrt \beta) \right) \cdot \sigma \sqrt{\frac{2 \log \frac{1}{\delta}}{n}}
    \end{align*}
    So, the claim follows.
\end{proof}

\lemtwodoutlierlightestimator*
\begin{proof}
    We will let $w \in \{0, 1\}$ be jointly distributed with $x$ such that $w = 1$ iff $x$ would not be thrown out in Algorithm~\hyperref[alg:two_dimensional_trimmed_mean_estimator]{\textsc{2DOutlierLightEstimator}}.

    \begin{itemize}
        \item 
        \textbf{$w = 0$ if $|\inner{u, x - \mu}|> 2 \sqrt{\beta} T$ for any $u \in \S^1$.}
        Since $|\mu_0^j - \inner{e_j, \mu}| \le O\left(\frac{\log \frac{1}{\delta}}{n} \right) \le \frac{\sqrt{\beta}T}{4}$, and since any sample with $|\inner{e_j, x_i} - \mu_0^j| > \sqrt{\beta} T$ is thrown out, $w = 0$ for any $x$ with $|\inner{e_j, x} - \inner{e_j, \mu}| > \frac{5 \sqrt{\beta} T}{4}$.
    
        Furthermore, for any $u \in \S^1$, we have that if $|\inner{u, x} - \inner{u, \mu}| > 2 \sqrt{\beta} T$, then, for some $j \in \{1, 2\}$,
        \begin{align*}
             |\inner{e_j, x} - \inner{e_j, \mu}| \ge \frac{1}{\sqrt{2}} |\inner{u, x} - \inner{u, \mu}| > \sqrt{2} \beta T > \frac{5 \sqrt{\beta}T}{4}
        \end{align*}
        so that $w = 0$ for any such $x$.
        So, $w$ satisfies that $w = 0$ if $|\inner{u, x - \mu}| > 2 \sqrt{\beta} T$.
        \item 
        \textbf{$\Pr[w=0] \le O\left(\frac{\log \frac{1}{\delta}}{n} \right)$.} Since $x$ is $(\beta, O(\beta))$-outlier-light, which means that $\E\left[|\inner{u, x - \mu}| \ge \beta T \right] \lesssim \beta \sigma^2$ for every $u \in \S^1$, we have that
        \begin{align*}
            \Pr\left[|\inner{u, x - \mu}| \ge 2 \sqrt{\beta} T \right] \le \Pr\left[|\inner{u, x - \mu}| \ge \beta T \right] \lesssim \frac{\beta \sigma^2}{\beta^2 T^2} \lesssim \frac{\log \frac{1}{\delta}}{n}
        \end{align*}
        So, since $|\mu_0^{e_j} - \inner{e_j, \mu}| \le O\left(\frac{\log \frac{1}{\delta}}{n} \right) \le \sqrt{\beta}T$, we have
        \begin{align*}
            \Pr\left[|\inner{e_j, x - \mu_0^{e_j}}| \ge \sqrt{\beta} T \right] \le \Pr\left[|\inner{e_j, x - \mu}| \ge 2 \sqrt{\beta} T \right] \lesssim \frac{\log \frac{1}{\delta}}{n}
        \end{align*}
        so that $\Pr[w = 0] \lesssim \frac{\log \frac{1}{\delta}}{n}$.
    \end{itemize}
    Now, let $U$ be a $\rho$-net as in Assumption~\ref{assumption:rho_net_U}, for $\rho = \delta^{\Theta(\sqrt{\beta})}$. By Lemma~\ref{lem:trimmed_mean_1d} and union bound, with probability $1 - \delta$, for every $u \in U$ simultaneously, and the estimate $\wh \mu$ returned by Algorithm~\hyperref[alg:two_dimensional_trimmed_mean_estimator]{\textsc{2DOutlierLightEstimator}},
    \begin{align*}
        |\inner{u, \wh \mu - \mu}| &\le \left(1 + O\left(\sqrt \beta \right) \right) \cdot \sigma \sqrt{\frac{2 \log \frac{2 |U|}{\delta}}{n}}\\
        &= \left( 1 + O(\sqrt{\beta})\right) \cdot \sigma \cdot \sqrt{\frac{2 \log \frac{2}{\delta}}{n}}
    \end{align*}
    Then, we have
    \begin{align*}
        \|\wh \mu - \mu\| &= \sup_{v : \|v\| = 1} \left|\inner{v, \wh \mu - \mu} \right|\\
        &\le \sup_{u \in U}\left|\inner{u, \wh \mu - \mu} \right| + \rho \|\wh \mu - \mu\|\\
        &\le \left(1 + O\left(\sqrt{\beta} \right) \right) \cdot \sigma \sqrt{\frac{2 \log \frac{2}{\delta}}{n}} + \delta^{\Theta(\sqrt{\beta})} \| \wh \mu - \mu\|
    \end{align*}
    so that
    \begin{align*}
        \|\wh \mu - \mu\| \le \left(1 + O\left(\sqrt \beta \right) \right)  \cdot \sigma \sqrt{\frac{2 \log \frac{2}{\delta}}{n}}
    \end{align*}
    as claimed.
\end{proof}

\subsection{Final Improved Two-Dimensional Estimator}
\setcounter{algorithm}{3}
\begin{algorithm}[H]
\caption{\textsc{2DHeavyTailedEstimator}}
    \vspace{0.2cm}
    \paragraph{Input parameters:}
    \begin{itemize}
        \item Failure probability $\delta$, Two-dimensional samples $x_1, \dots, x_n$, $\psi$ function, Scaling parameter $T$, Inlier-Outlier-Lightness parameters $\beta, L$, Approximation parameters $0< \xi, \tau < 1$, set of unit vectors $U$
    \end{itemize}
    \begin{enumerate}
        \item Using $\Theta(\xi) n$ samples, compute Median-of-Means estimates $\mu_0^{e_j}$ of the one-dimensional samples $\inner{e_j, x_i}$ with failure probability $\frac{\delta}{4(|U|+2)}$ for each $j \in \{1, 2\}$.
        \item Using $\Theta(\xi) n$ samples, compute Median-of-Means estimates $\mu_0^u$ of the one-dimensional samples $\inner{u, x_i}$ with failure probability $\frac{\delta}{4(|U|+2)}$ for each $u \in U$.
        \item Let the set of the remaining $(1 - \Theta(\xi)) n$ samples be $X'$. Run Algorithm~\hyperref[alg:2d_inlier_vs_outlier_light-tester]{\textsc{2DInlierOutlierLightTester}} using failure probability $\delta/4$, the samples in $X'$ and initial estimates $\mu_0^{e_1}, \mu_0^{e_2}$.
        \item If the output of \hyperref[alg:2d_inlier_vs_outlier_light-tester]{\textsc{2DInlierOutlierLightTester}} is some $e_j$, run \hyperref[alg:two-d-catoni]{\textsc{2DInlierLightEstimator}} using failure probability $\delta/8$, the samples in $X'$, and the initial estimates $\mu_0^u$, and output its mean estimate $\wh \mu$.
        \item If instead the output of \hyperref[alg:2d_inlier_vs_outlier_light-tester]{\textsc{2DInlierOutlierLightTester}} is $\perp$, run \hyperref[alg:two_dimensional_trimmed_mean_estimator]{\textsc{2DOutlierLightEstimator}} using failure probability $\delta/4$, the samples in $X'$ and initial estimates $\mu_0^{e_1}, \mu_0^{e_2}$. Return its output $\wh \mu$.
    \end{enumerate}
\end{algorithm}

\thmtwodestimator*
\begin{proof}
    First note that by classical results on Median-of-Means~\citep{mom_estimator} and a union bound, for every vector $v \in U \cup \{e_1, e_2\}$, we have with probability $1 - \delta/4$,
    \begin{align*}
        \left|\mu_0^v - \inner{v, \mu} \right| \le O\left(\sigma \sqrt{\frac{\log \frac{1}{\delta}}{n}} \right)
    \end{align*}
since $n > O_\xi( \log \frac{1}{\delta})$. For the remaining proof, we condition on the above. Now, by a union bound, there exist constants $0 < \beta, L < 1$, such that by Lemmas~\ref{lem:2d_inlier_outlier_testing},~\ref{lem:two-d-catoni}~and~\ref{lem:2d_trimmed_mean}, the following events happen with probability $1 - 3\delta/4$.
    \begin{itemize}
        \item If the output of Algorithm~\hyperref[alg:2d_inlier_vs_outlier_light-tester]{\textsc{2DInlierOutlierLightTester}} is $e_j$, $\inner{e_j, x_i}$ is $(\beta, L)$-inlier-light. On the other hand, if the output is $\perp$, $x_i$ is $(8 \beta, 8 L)$-outlier-light.
        \item If $\inner{e_j, x_i}$ is $(\beta, L)$-inlier-light, Algorithm~\hyperref[alg:two-d-catoni]{\textsc{2DInlierLightEstimator}} returns $\wh \mu$ with
        \begin{align*}
            \|\wh \mu - \mu\| &\le (1 - 2\tau + \Theta(\xi)) \cdot JUNG_2 \cdot \sigma \sqrt{\frac{2 \log \frac{2}{\delta}}{n}}\\
            &\le (1 - \tau)  \cdot JUNG_2 \cdot \sigma \sqrt{\frac{2 \log \frac{2}{\delta}}{n}}
        \end{align*}
        \item If $x_i$ is $(8\beta, 8L)$-outlier-light, for $L = O(\beta)$, Algorithm~\hyperref[alg:two_dimensional_trimmed_mean_estimator]{\textsc{2DOutlierLightEstimator}} returns $\wh \mu$ with
            \begin{align*}
                \| \wh \mu - \mu \| \le \left(1 + O\left(\sqrt{\beta} \right) \right) \cdot \sigma \sqrt{\frac{2 \log \frac{1}{\delta}}{n}}
            \end{align*}
    \end{itemize}
    So, with probability $1 - \delta$ in total, for $\beta$ small enough, 
    Algorithm~\hyperref[alg:final_two_dimensional_estimator]{\textsc{2DHeavyTailedEstimator}} returns estimate $\wh \mu$ with 
    \begin{align*}
        \| \wh \mu - \mu\| \le (1 - \tau) \cdot JUNG_2 \cdot \sigma \sqrt{\frac{2 \log \frac{2}{\delta}}{n}}
    \end{align*}
\end{proof}
%

\subsection{Improved $d$-Dimensional Estimator}
\label{appendix:d_dim_estimator}
\paragraph{Notation.} For $x \in \R^d$ and subspace $W$, we will let $x_{\| W}$ mean the projection of $x$ onto $W$.
\begin{assumption}
\label{assumption:subspace_net}
    $V \subset \left(\S^{d-1} \right)^2$ is a set of size $\left( \frac{1}{\zeta}\right)^{O(d)}$ of pairs of vectors $(v^1, v^2)$ with $W$ the subspace spanned by vectors in pair $j$. Let $W^\perp$ be the subspace orthogonal to $W$. Then, for any $x \in \R^d$, there exists $(v^1, v^2) \in V$ such that for the subspace $W$ spanned by $(v^1, v^2)$, $\|x_{\|W^\perp}\| \le \zeta \|x\|$.
\end{assumption}

Note that for every $\zeta < 1$ and $d$, there exists a set $V$ satisfying the above assumption. In particular, if we let $Z \subset \R^d$ be a $\zeta$-net of size $\left(\frac{1}{\zeta} \right)^{O(d)}$, and then let $V \subset (\R^d)^2$ be the set of pairs $(z, w)$ for each $z \in Z$ and any vector $w$ orthogonal to $z$, then $V$ satisfies Assumption~\ref{assumption:subspace_net}.
\setcounter{algorithm}{7}
\begin{algorithm}[H]\caption{\textsc{HighDimensionalHeavyTailedEstimator}}
    \label{alg:final_d_dimensional_estimator} 
    \paragraph{Input parameters:}
    \begin{itemize}
        \item Failure probability $\delta$, $d$-dimensional samples $x_1, \dots, x_n$, Covariance bound $\sigma^2 I_d$.
    \end{itemize}
    \begin{enumerate}
        \item Let $\beta, L, \xi, \tau, \zeta$ be sufficiently small universal constants as given by Theorem~\ref{thm:heavy_upper}~and~\ref{thm:2d_estimator}. Let $\psi$ be a function Assumption~\ref{assumption:improved_catoni_psi}. Let $T = \sigma \sqrt{\frac{\log \frac{2}{\delta}}{n}}$. Let $U \subset \R^2$ satisfy Assumption~\ref{assumption:rho_net_U} for $\rho = \delta^{\Theta(\xi)}$, and let $V \subset (\R^d)^2$ satisfy Assumption~\ref{assumption:subspace_net}.
        \item For each pair of vectors $(v^1, v^2) \in V$, let $W$ be the subspace spanned by them. Let $x_{\|W}$ be the projection of vector $x$ onto $W$.
        \item For each $(v^1, v^2) \in V$ with associated subspace $W$, run Algorithm~\ref{alg:final_two_dimensional_estimator} using samples $x_{1\| W}, \dots, x_{n\|W}$ with failure probability $\delta/|V|$, and approximation parameters $\xi, \Theta(\tau)$, and let the output be two-dimensional mean estimate $\wh \mu_{W}$.
        \item For each $(v^1, v^2) \in V$ with associated subspace $W$, consider the set $S_{(v^1, v^2)} = \left\{w : \|w_{\|W} - \wh \mu_{W}\| \le (1  - \Theta(\tau)) \cdot JUNG_2 \cdot  \sqrt{\frac{2 \log \frac{2}{\delta}}{n}} \right \}$. Let $S$ be the convex set given by $S:= \cap_{(v^1, v^2) 
        \in V} S_{(v^1, v^2)}$.
        \item Return the center $\wh \mu$ of the minimum enclosing ball of the set $S$ as the mean estimate.
    \end{enumerate}
\end{algorithm}

%
\heavyupperthm*

\begin{proof}
    By Theorem~\ref{thm:2d_estimator} and a union bound, with probability $1 - \delta$,
    \begin{align*}
        \|\wh \mu_{W} - \mu_{\| W}\| &\le (1 - \Theta(\tau) + \xi) \cdot JUNG_2 \cdot \sigma \sqrt{\frac{2 \log \frac{1|V|}{\delta}}{n}}\\
        &\le O\left(\sqrt{\frac{d \log \frac{1}{\zeta}}{n}}\right) + \left(1 - \Theta(\tau) +\xi \right) \cdot JUNG_2 \cdot \sigma \sqrt{\frac{2 \log \frac{1}{\delta}}{n}}\\
        &\le \left(1 - \Theta(\tau) \right) \cdot JUNG_2 \cdot \sigma \sqrt{\frac{2 \log \frac{1}{\delta}}{n}}
    \end{align*}
    since $n \ge C^2 d$. Thus, conditioned on the above, $\mu$ when projected onto any subspace $W$ spanned by $(v^1, v^2) \in V$, lies in $S$ projected onto $W$. So, by Theorem~\ref{thm:generalized_jung}, for the center $\wh \mu$ of the minimum enclosing ball of $S$, and any $W$ spanned by $(v^1, v^2) \in V$,
    \begin{align*}
        \|(\wh \mu - \mu)_{\|W}\| \le \left(1 - \Theta(\tau) \right) \cdot JUNG_d \cdot \sigma \sqrt{\frac{2 \log \frac{1}{\delta}}{n}}
    \end{align*}
    Then, by Assumption~\ref{assumption:subspace_net}, there exists $(v^1 v^2) \in V$ with associated subspace $W$ such that \[\|(\wh \mu - \mu)_{\|W^\perp}\| \le \zeta \|\wh \mu - \mu\|\] So,
    \begin{align*}
        \|\wh \mu - \mu\| &= \|(\wh \mu - \mu)_{\| W}\| + \|(\wh \mu - \mu)_{\|W^\perp}\|\\
        &\|(\wh \mu - \mu)_{\| W}\| + \zeta \cdot \| \wh \mu - \mu\|
    \end{align*}
    so that
    \begin{align*}
        \|\wh \mu - \mu\| \le (1 - \tau) \cdot JUNG_d \cdot \sigma \sqrt{\frac{2 \log \frac{2}{\delta}}{n}}
    \end{align*}
    for $\tau$ sufficiently small.
\end{proof}

\section{Robust Lower bound}\label{sec:robust-lower}


Let $v_1, \dotsc, v_{d+1}$ be the $d+1$ vertices of a regular
$d$-dimensional simplex centered at the origin, with $\norm{v_i} = 1$.
Then $\sum v_i = 0$ and $\inner{v_i, v_j} = -\frac{1}{d}$ for
$i \neq j$.

For $\eps \leq \frac{1}{d+1}$, define $D^*$ be the distribution that is
each $v_i$ with probability $\eps$, and $0$ with the remaining
probability $1-\eps(d+1)$ probability.  So $D^*$ has mean $0$ and
an isotropic variance of
\[
  \E_{x \sim D^*}[\inner{v_1, x}^2] = \eps \cdot 1 + (d \eps) \frac{1}{d^2} = \frac{d+1}{d} \eps.
\]

For each $j \in [d+1]$, let $D_j$ be the same as $D^*$ except
replacing $v_j$ with $-v_j$.  Then $D_j$ has mean $- 2 \eps v_j$.  For
every direction $u \perp v_j$, $D_j$ has the same variance
$\frac{d+1}{d} \eps$ as $D^*$; and the variance in direction $v_j$ is
\[
  \E_{x \sim D^*}[\inner{v_j, x}^2] - \E_{x \sim D^*}[\inner{v_j, x}]^2 = \frac{d+1}{d} \eps - 4 \eps^2 < \frac{d+1}{d} \eps.
\]

Thus each $D_j$ has covariance $\Sigma \preceq \frac{d+1}{d} \eps I$,
and $TV(D^*, D_j) = \eps$ for all $j$.

Informally, this means that robust mean estimation, on input
$(D^*, \sigma, \eps)$, needs to output a mean $\wh{\mu}$ that is good
for each $D_j$; the best it can do is output $0$, which has error
$2 \eps$ for each $i$.  Thus the error is
\[
  2 \eps = \sqrt{\frac{2d}{d+1}} \sqrt{2\norm{\Sigma} \eps}
\]
This constant, $\sqrt{\frac{2d}{d+1}}$, is $JUNG_d$.  More formally,
we start with this lemma:

\begin{lemma}\label{lem:simplexdist}
  Let $v_1, \dotsc, v_{d+1} \in \R^d$ be vertices of a regular
  simplex centered at the origin.  Then for any vector $u \in \R^d$.
  \[
    \E_{i \in [d+1]}[ \norm{v_i - u} ] \geq \norm{v_1}.
  \]
\end{lemma}
\begin{proof}
  We can write $u$ in barycentric coordinates, $u = \sum a_i v_i$ for
  $\sum a_i = 1$.  Then for any permutation $\pi$ of $[d+1]$, we write
  $u_{\pi} := \sum a_{\pi(i)} v_i$.  By symmetry, this satisfies
  \[
     \E_{i \in [d+1]}[ \norm{v_i - u_\pi} ] = \E_{i \in [d+1]}[ \norm{v_i - u} ].
   \]
   By choosing $\pi$ to be a uniform permutation,
   \[
     \E_{i \in [d+1]}[ \norm{v_i - u} ] =  \E_{\pi} \E_{i \in [d+1]}[ \norm{v_i - u_\pi} ] \geq  \E_{i \in [d+1]}[ \norm{v_i - \E_{\pi}[u_\pi]} ] = \E_{i \in [d+1]}[ \norm{v_i}] = \norm{v_1}.
   \]
   
\end{proof}

\begin{lemma}\label{lem:robust-lower-smalleps}
  For every $d \geq 1$ and $\eps \leq \frac{1}{d+1}$, every algorithm for robust estimation of
  $d$-dimensional distributions with covariance
  $\Sigma \preceq \sigma^2 I$ has error rate
  \[
    \E[\norm{\wh{\mu} - \mu}] \geq JUNG_d \cdot \sqrt{2 \sigma^2 \eps}
  \]
  on some input distribution.
\end{lemma}
\begin{proof}
  Take the distributions $D^*$, $D_j$ described above, so $\sigma^2 = \frac{d+1}{d} \eps$.  Suppose the
  true distribution is $D_j$ for a random $j \in [d+1]$, and the
  adversary perturbs each $D_j$ into $D^*$, then gives the adversary
  samples from $D^*$.  The algorithm's output $\wh{\mu}$ is
  independent of $j$, and has expected error
  \[
    \E[\norm{\wh{\mu} - \mu}]= \E_{\wh{\mu}, j}[\norm{\wh{\mu} - (- 2\eps v_j)}].
  \]
  By Lemma~\ref{lem:simplexdist}, this is at least $2 \eps$.  Thus
  \[
    \E[\norm{\wh{\mu} - \mu}] \geq 2 \eps = \sqrt{\frac{2 d}{d+1}}\sqrt{2 \sigma^2 \eps} = JUNG_d \sqrt{2 \sigma^2\eps}.
  \]
\end{proof}

Finally, we remove the restriction that $\eps \leq \frac{1}{d+1}$ by
applying the above lemma to $(1/\eps - 1)$-dimensional space.

\robustlowerthm*


\begin{proof}
  If $d \leq \frac{1}{\eps} - 1$, this is the same as
  Lemma~\ref{lem:robust-lower-smalleps}.  For
  $d > \frac{1}{\eps} - 1$, we instead restrict to a
  $d' = \floor{\frac{1}{\eps} - 1}$-dimensional space before
  applying Lemma~\ref{lem:robust-lower-smalleps}.  Thus
  \[
    \E[\norm{\wh{\mu} - \mu}] \geq JUNG_{d'} \cdot \sqrt{2 \sigma^2 \eps}
  \]
  Now,
  \[
    JUNG_{d'} = \sqrt{\frac{2d'}{d'+1}} = \sqrt{2} \sqrt{1 - \frac{1}{\floor{1/\eps}}} \geq \sqrt{2} \cdot (1 - \eps) \geq JUNG_d \cdot (1 - \eps).
  \]
\end{proof}

\section{Robust Estimation, Upper Bound}\label{sec:robust-upper}

The following result is folklore:

\begin{lemma}\label{lem:robust1d}
  If $X, Y$ are real-valued variables with $\Var(X), \Var(Y) \leq \sigma^2$ and $TV(X, Y) \leq 2\eps$, then

  \[
    \E[X] - \E[Y] \leq \frac{2  \sqrt{2\sigma^2 \eps}}{\sqrt{1-2\eps}}
  \]
\end{lemma}
\begin{proof}
  Couple $X$ and $Y$ so that $\Pr[X \neq Y] \leq 2 \eps$.  Then by Cauchy-Schwarz,
  \begin{align*}
    \E[X - Y]^2 &= \E[ (X - Y) 1_{X \neq Y}]\\
                &\leq \E[ (X - Y)^2] \E[1_{X \neq Y}^2]\\
                &\leq (\Var(X - Y) + \E[X - Y]^2)\cdot 2\eps.
  \end{align*}
  Canceling terms, and using that $\Var(X - Y) \leq 2 (\Var(X)  + \Var(Y)) \leq 4 \sigma^2$,
  \[
    (1 - 2 \eps) \E[X - Y]^2 \leq 8 \sigma^2 \eps
  \]
  giving the result.
\end{proof}

\robustupperthm*

\begin{proof}
  Given the corrupted input distribution $D'$, take the set of all
  possible distributions $X$ with $TV(X, D) \leq \eps$ and
  $\Var(X) \leq \sigma^2$, and look at the corresponding means.  Let
  $S$ denote the set of these candidate means.  We know that the
  uncorrupted distribution lies in the candidate set, so its mean lies
  in $S$.

  For any two distributions $X, Y$ in the candidate set, we have
  $TV(X, Y) \leq TV(X, D) + TV(D, Y) \leq 2 \eps$.  Therefore the
  same holds for any 1-dimensional projections $\inner{v, X}$; in
  particular, by Lemma~\ref{lem:robust1d},
  \[
    \norm{\E[X] - \E[Y]} = \max_{\norm{v} = 1} \E[\inner{v, X} - \inner{v, Y}] \leq \frac{2  \sqrt{2\sigma^2 \eps}}{\sqrt{1-2\eps}}
  \]
  so $S$ has diameter at most $\frac{2  \sqrt{2\sigma^2 \eps}}{\sqrt{1-2\eps}}$.

  Then Jung's theorem states that the circumcenter of $S$ has distance
  at most $JUNG_d \cdot \frac{\sqrt{2\sigma^2 \eps}}{\sqrt{1-2\eps}}$
  to each point in $S$, and in particular to the true mean.  Finally,
  given that $\eps \leq 0.3$,
  $\frac{1}{\sqrt{1-2\eps}} < 1 + 2\eps$.
\end{proof}

\section{Geometry Results}
\begin{theorem}[Jung's Theorem \citep{Jung1901}]
    Let $K \subset \R^d$ be a compact set and let $D = \max_{p, q \in K} \|p - q\|_2$ be the diameter of $K$. There exists a closed ball with radius
    \begin{align*}
        R \le D \sqrt{\frac{d}{2(d+1)}}
    \end{align*}
    that contains $K$. The boundary case of equality is obtained by the $d$-simplex.
\end{theorem}

\begin{theorem}[Generalized Jung's Theorem \citep{generalized_Jung}]
    \label{thm:generalized_jung}
    Let $K \subset \R^d$ be a compact set, and let $R_i$ be the maximum circumradius of any $i$-dimensional projection of $K$. Then, for any $1 \le j \le i \le d$,
    \begin{align*}
        R_i \le \sqrt{\frac{i(j+1)}{j(i+1)}} \cdot R_j
    \end{align*}
\end{theorem}

\end{document}